\documentclass[11pt]{amsart} 
\usepackage[vmargin=2cm,hmargin=2cm]{geometry}
\usepackage{kpfonts}
\usepackage[utf8x]{inputenc}

\usepackage{enumerate} 
\usepackage[hidelinks]{hyperref}
\usepackage{amsfonts,amsmath,amssymb}

\usepackage{multirow}
\usepackage{float}

\usepackage[norelsize,boxed,vlined]{algorithm2e}

\SetKwComment{Comment}{\# }{}
\SetKwInOut{Input}{Input}
\SetKwInOut{Output}{Output}
\SetKwFunction{Add}{Add}
\SetKwFunction{AddSet}{AddSet}
\SetKwFunction{Difference}{Difference}
\SetKwFunction{Filtered}{Filtered}
\SetKwFunction{First}{First}
\SetKwFunction{ForAll}{ForAll}
\SetKwFunction{ForAny}{ForAny}
\SetKwFunction{Intersection}{Intersection}
\SetKwFunction{IsRange}{IsRange}
\SetKwFunction{Length}{Length}
\SetKwFunction{List}{List}
\SetKwFunction{Maximum}{Maximum}
\SetKwFunction{Minimum}{Minimum}
\SetKwFunction{RandomList}{RandomList}
\SetKwFunction{ShallowCopy}{ShallowCopy}
\SetKwFunction{Set}{Set}
\SetKwFunction{Union}{Union}
\SetKwData{Integers}{Integers}
\SetKwFunction{Not}{not}
\SetKwFunction{Or}{ or }
\SetKwFunction{And}{ and }
\SetKwData{fail}{fail}
\SetKwData{false}{false}
\SetKwData{true}{true}
\SetKwData{xbreak}{break}
\SetKwFunction{NumericalSemigroupByGaps}{NumericalSemigroupByGaps}
\SetKwFunction{NumericalSemigroup}{NumericalSemigroup}
\SetKwFunction{RepresentsGapsOfNumericalSemigroup}{RepresentsGapsOfNumericalSemigroup}
\SetKwFunction{Divisors}{Divisors}
\SetKwFunction{PositiveInt}{PositiveInt}
\SetKwFunction{closure}{\ref{func:Closure}}
\SetKwFunction{SmallElements}{smalls}
\SetKwFunction{startingForcedGaps}{\ref{func:StartingForcedGaps}}
\SetKwFunction{furtherForcedElements}{\ref{func:FurtherForcedElements}}
\SetKwFunction{furtherForcedGaps}{\ref{func:FurtherForcedGaps}}
\SetKwFunction{nonAdmissible}{\ref{func:NonAdmissible}}
\SetKwFunction{endingCondition}{\ref{func:EndingCondition}}
\SetKwFunction{recursiveDepthFirstExploreTree}{\ref{func:RecursiveDepthFirstExploreTree}}
\SetKwFunction{Closurew}{Closure}
\SetKwFunction{StartingForcedGapsw}{StartingForcedGaps}
\SetKwFunction{FurtherForcedElementsw}{FurtherForcedElements}
\SetKwFunction{FurtherForcedGapsw}{FurtherForcedGaps}
\SetKwFunction{NonAdmissiblew}{NonAdmissible}
\SetKwFunction{EndingConditionw}{EndingCondition}
\SetKwFunction{RecursiveDepthFirstExploreTreew}{Recursive\-Depth\-FirstExploreTree}
\SetKwFunction{SimpleForcedIntegers}{Simple\-Forced\-Integers}
\SetKwFunction{ForcedIntegers}{ForcedIntegers}
\SetKwFunction{ForcedIntegersQV}{ForcedIntegers\_QV}
\SetKwFunction{PFfunction}{PF}
\SetKw{semigroups}{semigroups}
\SetKwData{frob}{frob}
\SetKwData{type}{type}
\SetKwData{fes}{e}
\SetKwData{fgs}{g}
\SetKwData{elts}{elts}
\SetKwData{fe}{fe}
\SetKwData{fg}{fg}
\SetKwData{PF}{PF}

\usepackage{pgf} 
\usepackage{tikz} \usepgfmodule{plot}
\usepgflibrary{plothandlers} 
\usetikzlibrary{shapes.geometric}
\usetikzlibrary{arrows}
\usetikzlibrary{graphs}
\usetikzlibrary{fit}

\usepackage{graphicx}
\newtheorem{theorem}{Theorem} 
\newtheorem{lemma}[theorem]{Lemma}
\newtheorem{corollary}[theorem]{Corollary}
\newtheorem{proposition}[theorem]{Proposition}

\theoremstyle{remark} 
\newtheorem{example}[theorem]{Example}
\newtheorem{remark}[theorem]{Remark}
\newtheorem{question}[theorem]{Question}

\DeclareMathOperator{\Goper}{gaps} 
\DeclareMathOperator{\SEoper}{smalls} 
\DeclareMathOperator{\PFoper}{PF} 
\DeclareMathOperator{\toper}{t} 
\DeclareMathOperator{\calGFoper}{\mathcal{GF}} 
\DeclareMathOperator{\calEFoper}{\mathcal{EF}} 
\DeclareMathOperator{\calSoper}{\mathcal{S}} 

\newcommand{\PFsf}{\ensuremath{\mathsf{PF}}}
\newcommand{\typesf}{\ensuremath{\mathsf{type}}}
\newcommand{\frobsf}{\ensuremath{\mathsf{frob}}}

\title[Algorithm for pseudo-Frobenius]
{Numerical semigroups with a given set of pseudo-Frobenius numbers}

\author{M. Delgado} 
\address{CMUP, Departamento de Matem\'atica, Faculdade de
  Ci\^encias, Universidade do Porto, Rua do Campo Alegre 687, 4169-007 Porto,
  Portugal} 
\email{mdelgado@fc.up.pt} 
\thanks{The first author was partially supported by CMUP (UID/MAT/00144/2013), which is funded by FCT (Portugal) with national (MEC) and European structural funds through the programs FEDER, under the partnership agreement PT2020. He also acknowledges the hospitality of the University of Granada during the various visits made, in particular one supported by the GENIL SSV program}

\author{P. A. Garc\'{\i}a-S\'{a}nchez} 
\address{Departamento de \'Algebra,
  Universidad de Granada, 18071 Granada, Espa\~na} 
\email{pedro@ugr.es}

\author{A. M. Robles-P\'erez} 
\address{Departamento de Matem\'atica Aplicada,
  Universidad de Granada, 18071 Granada, Espa\~na} 
\email{arobles@ugr.es}

\thanks{The second and third authors are supported by the projects
  MTM2010-15595, FQM-343 and FEDER funds. The third author is also supported by the Plan Propio (Plan de Fortalecimiento de los Grupos de Investigaci\'on) of the Universidad de Granada}

\keywords{numerical semigroups, algorithms, pseudo-Frobenius numbers, irreducible numerical semigroups}

\subjclass[2010]{11D07, 20M14, 20--04}

\begin{document}
\begin{abstract}
The pseudo-Frobenius numbers of a numerical semigroup are those gaps of the numerical semigroup that are maximal for the partial order induced by the semigroup.
We present a procedure to detect if a given set of integers is the set of pseudo-Frobenius numbers of a numerical semigroup and, if so, to compute the set of all numerical semigroups having this set as set of pseudo-Frobenius numbers.
\end{abstract}
\maketitle
\section{Introduction}
\label{sec:introduction}
Let $S$ be a numerical semigroup, that is, a cofinite submonoid of $(\mathbb N,+)$, where $\mathbb N$ stands for the set of nonnegative integers.

An integer $x$ is said to be the \emph{Frobenius number} of $S$ (respectively,
a \emph{pseudo-Frobenius number} of $S$) if $x\not\in S$ and $x+s\in S$, for all $s\in \mathbb{N}\setminus{\{0\}}$ (respectively, for all $s\in S\setminus{\{0\}}$).

Given a positive integer $f$, there exist numerical semigroups whose Frobenius number is $f$. One example of such a semigroup is the semigroup $\{0,f+1,\to\}$ containing $0$ and all the integers greater than $f$.
There are several algorithms to compute all the numerical semigroups with a given Frobenius number (the fastest we know is based on \cite{huecos-fun}).

We denote by $\mathrm{F}(S)$ the Frobenius number of $S$ and by $\PFoper(S)$ the set of pseudo-Frobenius numbers of $S$. The cardinality of $\PFoper(S)$ is said to be the \emph{type} of $S$ and is denoted by $\toper(S)$.

A positive integer that does not belong to $S$ is said to be a \emph{gap} of $S$ and an element of $S$ that is not greater than $\mathrm{F}(S)+1$ is said to be a \emph{small element} of $S$. To denote the set $\mathbb{N}\setminus S$ of gaps of $S$ we use $\Goper(S)$ and to denote the set of small elements of $S$ we use $\SEoper(S)$. Since a set of gaps must contain the divisors of all its members and a set of small elements must contain all multiples of its members (up to its maximum), it is clear that there are sets of positive integers that can not be the set of gaps or the set of small elements of a numerical semigroup. The set of gaps, as well as the set of small elements, completely determines the semigroup. Observe that when some elements or some gaps are known, others may be forced. For instance, a gap forces all its divisors to be gaps. 
\medskip

Let $n$ be a positive integer and let $\PFsf = \{g_1,g_2,\ldots,g_{n-1}, g_n\}$ be a set of positive integers.
Denote by $\calSoper(\PFsf)$ the set of numerical semigroups whose set of pseudo-Frobenius numbers is \PFsf. When $n>1$, $\calSoper(\PFsf)$ may clearly be empty. Moreover, when non-empty, it is finite. In fact, $\calSoper(\PFsf)$ consists of semigroups whose Frobenius numbers is the maximum of \PFsf.
 Some questions arise naturally. Among them, we can consider the following.
\begin{question}\label{quest:existence}
Find conditions on the set \PFsf\ that ensure that $\calSoper(\PFsf)\ne \emptyset$.
\end{question}
\begin{question}\label{quest:compute_all}
Find an algorithm to compute $\calSoper(\PFsf)$.
\end{question} 
Both questions have been solved for the case that the set \PFsf\ consists of a single element (which must be the Frobenius number of a numerical semigroup; symmetric numerical semigroups) or when \PFsf\ consists on an even positive integer $f$ and $f/2$ (pseudo-symmetric numerical semigroups), see \cite{br}.  
Moreover, Question~\ref{quest:existence} was solved by Robles-P\'erez and Rosales~\cite{RoblesRosales_preprint:type2} in the case where \PFsf\ consists of $2$ elements (not necessarily of the form $\{f,f/2\}$).

The set $\calSoper(\PFsf)$ can be computed by filtering those semigroups that have \PFsf\ as set of pseudo-Frobenius numbers, from the numerical semigroups whose Frobenius number is the maximum of \PFsf\ (cf. Example~\ref{example:first}). Due, in part, to the possibly huge number of semigroups with a given Frobenius number, this is a rather slow procedure and we consider it far from being a satisfactory answer to Question~\ref{quest:compute_all}. 

Irreducible numerical semigroups with odd  Frobenius number correspond with symmetric numerical semigroups, and those with even Frobenius number with pseudo-symmetric (see for instance  \cite[Chapter 3]{NS}).  Bresinsky proved in \cite{bres-sym} that symmetric numerical semigroups with embedding dimension four have minimal presentations of cardinality 5 or 3 (complete intersections). Symmetry of a numerical semigroup $S$ translates to having $\{\mathrm F(S)\}$ as set of pseudo-Frobenius numbers. Later Komeda, \cite{komeda}, was able to prove the same result for pseudo-symmetric numerical semigroups (though he used different terminology for this property; in this setting 3 does not occur since pseudo-symmetric are never complete intersections). A numerical semigroup $S$ is pseudo-symmetric if its set of pseudo-Frobenius numbers is $\{\mathrm F(S), \mathrm F(S)/2\}$. It should be interesting to see the relationship with the type and the cardinality of a minimal presentation, and thus having tools to find semigroups with given sets of pseudo-Frobenius numbers becomes helpful. Watanabe and his students Nari and Numata are making some progress in the study of this relationship.

\subsection{Contents}
\label{subsec:organization}

We present two different procedures to determine the set of all numerical semigroups with a given set of pseudo-Frobenius numbers. One exploits the idea of irreducible numerical semigroup. From each irreducible numerical semigroup we start removing minimal generators with certain properties to build a tree whose leafs are the semigroups we are looking for. The other approach is based on determining the elements and gaps of any numerical semigroup with the given set of pseudo-Frobenius numbers, obtaining in this way a list of ``free'' integers. We then construct a binary tree in which branches correspond to assuming that these integers are either gaps or elements.

We start this work with some generalities and basic or well known results and connections with the \texttt{GAP}~\cite{GAP4} package \texttt{numericalsgps}~\cite{numericalsgps} (Sections~\ref{sec:generalities} and~\ref{sec:connections_numericalsgps}). Then we describe a procedure to compute forced integers (Sections~\ref{sec:preliminary_forced_ints} and~\ref{sec:computing_forced_ints}). As computing forced integers is fast and leads frequently to the conclusion that there exists no semigroup fulfilling the condition of having the given set as set of pseudo-Frobenius numbers, this approach benefits from the work done in the preceding sections.  
The following section of the paper is devoted to the above-mentioned approach of constructing a tree with nodes lists of integers, which turns out to be faster most of the times than the one based on irreducible numerical semigroups. Nevertheless, besides being useful to compare results, the method using irreducible numerical semigroups is of theoretical importance so we decided to keep it, and it is given in Appendix~\ref{sec:computing_arbitrary_type}.
The next section describes an algorithm that (increasing the number of attempts, if necessary) returns one numerical semigroup with the given set of pseudo-Frobenius numbers, when such a semigroup exists. Although it may be necessary to increase the number of attempts, usually does not require the computation of all the semigroups fulfilling the condition and is reasonably fast in practice (Section~\ref{sec:random}).

We give pseudo-code for the various algorithms that are used in the approach of lists of free integers. The pseudo-code given has an immediate translation to the \textsf{GAP} language, which is the programming language we used to implement these algorithms. We observe that it is a high level programming language. Note that we take advantage of the existence of the \textsf{GAP} package \textsf{numericalsgps} for computing with numerical semigroups. The names used for the functions described here are slightly different from the ones used in the package, since there longer names are required to be accurate with the variable names policy in \textsf{GAP}. Essentially, it will be a matter of adding a suffix to the names used for the package.

Many examples are given throughout the paper, some to illustrate the methods proposed, while others are included to motivate the options followed. In some of the examples we show the output obtained in a \textsf{GAP} session. These usually correspond to examples that are not suitable to a full computation just using a pencil and some paper; furthermore, indicative running times, as given by \textsf{GAP}, are shown (mainly in Section~\ref{sec:running_times}).

A new version of the \textsf{numericalsgps} package, including implementations of the algorithms developed in the present work, is released at the same time this work is made public. The implementations can be checked (the software in question is open source; see the links in the references) and can be used for testing examples.

\section{Generalities and basic results}
\label{sec:generalities}
Throughout the paper we will consider often a set $\PFsf= \{g_1,g_2,\ldots,g_{n-1},g_n\}$ of positive integers. We will usually require it to be ordered, that is, we will assume that $g_1<g_2<\cdots <g_{n-1}<g_n$. 
For convenience, we write  $\PFsf= \{g_1<g_2<\cdots <g_{n-1}<g_n\}$ in this case.

We denote by \textsf{frob} the maximum of \PFsf\ and by \typesf\ the cardinality of \PFsf. Note that if $S\in \calSoper(\PFsf)$, then $\frobsf=g_n=\mathrm{F}(S)$ and $\typesf=n=\toper(S)$.

\subsection{Forced integers}
\label{subsec:forced_integers}
We say that an integer $\ell$ is a \emph{gap forced by} \PFsf\ or a \PFsf-\emph{forced gap} if $\ell$ is a gap of all the numerical semigroups having \PFsf\ as set of pseudo-Frobenius numbers. In particular, if there is no semigroup with \PFsf\ as set of pseudo-Frobenius numbers, then every nonnegative integer is a gap forced by \PFsf. We use the notation $\calGFoper(\PFsf)$ to denote the set of \PFsf-forced gaps. 
In symbols: $\calGFoper(\PFsf)= \bigcap_{S\in \calSoper(\PFsf)}\Goper(S)$.

In a similar way, we say that an integer $\ell$ is an \emph{element forced by} \PFsf\ or a
\PFsf-\emph{forced element} if $\ell$ is an element of all semigroups in $\calSoper(\PFsf)$. We use the notation $\calEFoper(\PFsf)$ to denote the set of (small) \PFsf-forced elements. 
In symbols: $\calEFoper(\PFsf)= \bigcap_{S\in \calSoper(\PFsf)}\SEoper(S)$. Note also that if $\calSoper(\PFsf)=\emptyset$, then $\calEFoper(\PFsf) = \mathbb{N}$.

The union of the \PFsf-forced gaps and \PFsf-forced elements is designated by \PFsf-\emph{forced integers}.

The following is a simple, but crucial observation.
\begin{proposition}\label{prop:intersection_gaps_elements}
$\calSoper(\PFsf)\ne \emptyset$ if and only if $\calGFoper(\PFsf)\cap \calEFoper(\PFsf)=\emptyset$.
\end{proposition}
\begin{proof}
If $\calSoper(\PFsf)= \emptyset$, then all nonnegative integers are at the same time gaps and elements forced by \PFsf. Conversely, assume that $\calSoper(\PFsf)\ne \emptyset$ and let $S\in \calSoper(\PFsf)$. Then $\calGFoper(\PFsf)\cap\calEFoper(\PFsf)\subseteq \Goper(S)\cap \SEoper(S) = \emptyset$.
\end{proof}

Frequently the prefix \PFsf\ is understood and we will abbreviate by saying just \emph{forced gap}, \emph{forced element} or \emph{forced integer}. 

Let $G$ and $E$ be, respectively, sets of forced gaps and forced elements. The elements $v\in \{1,\ldots,\frob\}$ that do not belong to $G\cup E$ are said to be \emph{free integers for} $(G,E)$. When the pair $(G,E)$ is understood, we simply call \emph{free integer} to a free integer for $(G,E)$. 

\subsection{Well known results}
\label{subsec:general_results}
The partial order $\le_S$ induced by the numerical semigroup $S$ on the integers is defined as follows: $x\le_S y$ if and only if $y-x\in S$.
The following result is well known and will be used several times throughout this paper. 
\begin{lemma}\cite[Lemma~2.19]{NS}\label{lemma:maximals}
Let $S$ be a numerical semigroup. Then
\begin{enumerate}[(i)]
\item $\PFoper(S)= \mathrm{Maximals}_{\le_S}(\mathbb{Z}\setminus{S})$,
\item $x\in \mathbb{Z}\setminus{S} \mbox{  if and only if } f-x\in S \mbox{ for some } f\in \PFoper(S)$.
\end{enumerate}
\end{lemma}
It is well known that the type of a numerical semigroup $S$ is upper bounded by the least positive element belonging to $S$, which is known as the \emph{multiplicity} of $S$ and it is denoted by $\mathrm{m}(S)$.
\begin{lemma}\cite[Corollary~2.23]{NS}\label{lemma:type_multiplicity}
	Let $S$ be a numerical semigroup. Then
	$\mathrm{m}(S)\ge \mathrm{t}(S)+1$.
\end{lemma}

\subsection{Initializing the set of forced gaps}
\label{subsec:starting_forced_gaps}
The maximality of the pseudo-Frobenius numbers of $S$ with respect to $\le_S$ means that they are incomparable with respect to this ordering. In particular, the difference of any two distinct pseudo-Frobenius numbers does not belong to $S$, that is, 
\begin{equation}\label{eq:differences}
\{g_i-g_j\mid i,j\in \{1,\ldots,\toper(S)\}, i> j\} \subseteq \Goper(S).
\end{equation}
This is the underlying idea of the next result. 
\begin{lemma}\label{lemma:pseudo-comb-pseudo}
Let $S$ be a numerical semigroup and suppose that $\PFoper(S)=\{g_1<g_2<\cdots <g_{n-1}<g_n\}$, with $n>1$.
Let $i\in\{2,\ldots,\toper(S)\}$ and $g\in\langle \mathrm{PF(S)}\rangle$ with $g<g_i$. Then $g_i-g\in \Goper(S)$.
\end{lemma}
\begin{proof}
Assume that $g=g_{i_1}+\dots+g_{i_k}$ for some $k\in \mathbb N$. We proceed by induction on $k$. 
The case $k=1$ is given by (\ref{eq:differences}).
Assume that the result holds for $k-1$ and let us prove it for $k$. If $g_i-g\in S$, then $g_{i_k}+(g_i-g)\in S$ by definition of pseudo-Frobenius number. It follows that $g_i-(g_{i_1}+\dots+g_{i_{k-1}})\in S$, contradicting the induction hypothesis.
\end{proof}
\begin{remark}\label{rem:type_multiplicity}
Lemma~\ref{lemma:type_multiplicity} implies that $\{x\in\mathbb{N}\mid 1\le x\le \toper(S)\}\subseteq \Goper(S)$. Hence $\{1,\ldots, n\}\subseteq \calGFoper(\PFsf)$.
\end{remark}
As the pseudo-Frobenius numbers of $S$ are gaps of $S$ and any positive divisor of a gap must be a gap also, we conclude that the set of divisors of
$$\PFoper(S)\cup\{x\in\mathbb{N}\mid 1\le x\le \mathrm{t}(S)\}\cup\{g_i-g\mid i\in \{2,\ldots,\mathrm{t}(S)\}, g\in\langle \PFoper(S)\rangle, g_i>g\}$$
consists entirely of gaps of $S$.

Consider the set
$$\PFsf\cup\{x\in\mathbb{N}\mid 1\le x\le n\}\cup\{g_i-g\mid i\in \{2,\ldots,n\}, g\in\langle \PFsf\rangle, g_i>g\}$$
and denote by $\mathrm{sfg}(\PFsf)$ the set of its divisors (as we are only considering positive divisors, in what follows we will not include this adjective).
If $S$ is a numerical semigroup such that $\PFoper(S)=\PFsf$, we deduce that $\mathrm{sfg}(\PFsf)\subseteq \Goper(S)$. 
We have proved the following result for the case where there is a numerical semigroup $S$ such that $\PFoper(S)=\PFsf$. If no such semigroup exists, then $\calGFoper(\PFsf)=\mathbb N$ and the result trivially holds.
\begin{corollary}\label{cor:starting_forced_gaps}
Let $\PFsf$ be a set of positive integers. Then $\mathrm{sfg}(\PFsf)\subseteq \calGFoper(\PFsf)$.
\end{corollary}
We use the terminology \emph{starting forced gap for} \PFsf\ to designate any element of $\mathrm{sfg}(\PFsf)$, since $\mathrm{sfg}(\PFsf)$ is the set we start with when we are looking for forced gaps. In Subsection~\ref{subsec:forced_gaps_pseudo_code} we provide pseudo-code for a function to compute starting forced gaps.

\subsection{Initial necessary conditions}
\label{subsec:starting_conditions}
Let $n>1$ be an integer and let $\PFsf = \{g_1<\cdots <g_n\}$ be a set of positive integers.
\begin{lemma}\label{lemma:naive_condition}
Let $S$ be a numerical semigroup such that $\PFoper(S)=\PFsf$.
Let  $i\in\{2,\ldots,n\}$ and $g\in\langle \PFsf\rangle\setminus\{0\}$ with $g<g_i$. Then there exists $k\in \{1,\ldots,n\}$ such that $g_k-(g_i-g)\in S$. 
\end{lemma}
\begin{proof}
Lemma~\ref{lemma:pseudo-comb-pseudo} assures that $g_i-g\not\in S$. The conclusion follows from Lemma~\ref{lemma:maximals}.
\end{proof}

By choosing $i=n$ in the above result, there exists $k\ne n$ such that $g_k-(g_n-g_1)\ge 0$ and $g_k-(g_n-g_1)\not\in\{1,\ldots,\typesf\}$. (Note that $k=n$ would imply $g_1\in S$, which is impossible.) But then $g_{n-1}-(g_n-g_1)\ge 0$, since $g_{n-1}\ge g_k$ for all $k\in \{1,\ldots,n-1\}$. We have thus proved the following corollary.
\begin{corollary}\label{cor:naive_condition_g1}
Let $S$ be a numerical semigroup such that $\PFoper(S)=\PFsf$. Then  $g_1\ge g_n-g_{n-1}$. 
\end{corollary}

The computational cost of testing the condition $g_1\ge g_n-g_{n-1}$ obtained in Corollary~\ref{cor:naive_condition_g1} is negligible and should be made before calling any procedure to compute $\calSoper(\PFsf)$, avoiding in many cases an extra effort that would lead to the empty set. 

Other conditions of low computational cost would also be useful. Since $g_{n-1}-(g_n-g_1)\ge 0$, one could be tempted to ask whether replacing $g_1$ by $g_2$ one must have $g_{n-1}-(g_n-g_2)\not\in\{1,\ldots,\typesf\}$ (since $\{1,\ldots,\typesf\}$ consists of gaps). The following example serves to rule out this one that seems to be a natural attempt.
\begin{example} Let $S=\langle 8, 9, 10, 11, 13, 14, 15\rangle$. One can check easily that $\PFoper(S)=\{5, 6, 7, 12\}$. For $\PFsf = \{5, 6, 7, 12\}$, we have $g_{n-1}-(g_n-g_2)=7-(12-6)=1\in \{1,\ldots,4\}=\{1,\ldots,\typesf\}$.
\end{example}
Later, in Subsection~\ref{subsec:condition_based_on_forced_integers}, we give an extra condition, which is based on forced integers. 

\section{Connections with \textsf{GAP} and the package \textsf{numericalsgps}}
\label{sec:connections_numericalsgps}
When developing the present algorithms we had one implementation in mind. As already referred, it was to be made in the \textsf{GAP} programming language, benefiting also of the package \textsf{numericalsgps}. In fact the implementation accompanied the development of the algorithms, with mutual benefits.

To understand the paper no previous familiarity with \textsf{GAP} is assumed.  We give some examples to get the familiarity needed to fully understand the pseudo-code. 
Many of the terms used for the pseudo-code presented are borrowed from the \textsf{GAP} programming language, which is a high level one. 
The terms
\emph{Union, Length, Difference}
are used with the meaning they have in \textsf{GAP}, which is clear. 
\smallskip

Like what is done in \textsf{GAP}, for an unary function that applied to an argument $arg$ returns an expression $expr$, we write 
$$arg \to expr.$$
This shorthand for writing a function is used in our pseudo-code. Next we briefly explain the meaning of some of the terms used. For a not so brief explanation, complete definitions and plenty of examples, \textsf{GAP}'s manual should be consulted.

 Here $list$ is a list (of integers or of numerical semigroups) and $func$ is an unary function (that applies to the objects in $list$).
\begin{itemize} \itemsep2pt
\item $\List( list, func)$ returns a new list $new$ of the same length as $list$ such that $new[i] = func(list[i])$.
\item $\Filtered( list, func )$
returns a new list that contains those elements of $list$ for which $func$ returns \emph{true}. The order of the elements in the result is the same as the order of the corresponding elements of this list.
\item One can use $\Set(list)$ (which is a synonym of $\mathrm{SSortedList}$ (``strictly sorted list'')) to get a list that is duplicate free and sorted (for some order).
\item $\AddSet(list,el)$ adds the element $el$ to the set $list$.
\item $\First( list, func )$ returns the first element of $list$ for which the unary function $func$ returns \emph{true}. If func returns false for all elements of list, then $\mathrm{First}$ returns \emph{fail}.
\item $\ForAll( list, func )$
tests whether $func$ returns \emph{true} for all elements in  $list$.
\item $\ForAny( list, func )$ tests whether $func$ returns \emph{true} for at least one element in $list$.
\item $\IsRange(list)$ detects if the argument is an interval of integers.
\end{itemize}

We also use the following abbreviations. Let $X$ be a list of integers and $P$ a list of positive integers.
\begin{itemize} \itemsep2pt
\item \texttt{PosInt(X)}, an abbreviation of \texttt{Filtered(X,IsPosInt)}, returns the positive integers of $X$.
\item \texttt{Divisors(P)}, an abbreviation of \texttt{Union(List(P,DivisorsInt))}, returns the divisors of the elements of $P$: it first computes the lists of divisors for each element in \texttt{P}, and then takes the union of all of them.
\end{itemize}

The \textsf{numericalsgps} package also influences our pseudo-code. In some cases we use directly the names of the available functions, but in some cases we use shorthands, which are intended to turn the pseudo-code more readable.

One of the functions we use is \RepresentsGapsOfNumericalSemigroup, which, for a given set $X$ of positive integers returns \emph{true} in case there exists a numerical semigroup $S$ such that $\Goper(S)=X$, and returns \emph{false} otherwise. 

The functions to produce numerical semigroups used here are \NumericalSemigroup, which is used when generators are given, and \NumericalSemigroupByGaps, which, for an input $X$, returns a numerical semigroup whose set of gaps is $X$, when such a semigroup exists.    

As a shorthand for $\mathrm{SmallElements}(S)$, which gives the elements of the numerical semigroup $S$ that are nor greater than $\mathrm{F}(S)+1$, we use simply  \SmallElements, which agrees with the notation already introduced.
\texttt{Closure(elts,frob)} is a shorthand for
\begin{center}
\texttt{NumericalSemigroup(Union(elts,[frob+1..frob+Minimum(elts)]))}
\end{center}
which gives the least numerical semigroup containing the set $\mathrm{elts}$ of positive integers and having the largest possible Frobenius number not greater than $\mathrm{frob}$.
Note that the minimum of $\mathrm{elts}$ is greater than or equal to the multiplicity of the semigroup. Therefore, the union considered ensures that the semigroup contains all integers that are greater than the number $\mathrm{frob}$ given (aiming to be an upper bound for the Frobenius number of the semigroup constructed, although in cases where the elements $\mathrm{elts}$ define a numerical semigroup, it may have smaller Frobenius number).

As in \textsf{GAP}, comments start with the character $\#$.

As an example on how we will present function/algorithms in this manuscript, we write the function \closure.
\begin{function}[ht]\caption{Closure()\label{func:Closure}}
\Closurew{\elts,\frob}\Comment*[f]{the least numerical semigroup containing \elts and} \\
\Comment*[f]{whose Frobenius number is not greater than \frob}\\
\Return \NumericalSemigroup(\Union(\elts,[\frob+1..\frob+\Minimum(\PositiveInt(\elts))]));
\end{function}

The following example is just a \textsf{GAP} session that is intended to illustrate how to compute the set $\calSoper(\PFsf)$ by filtering those semigroups that have \PFsf\ as set of pseudo-Frobenius numbers, from the numerical semigroups whose Frobenius number is the maximum of \PFsf. This process was mentioned in the introduction.

Throughout the examples, we use a simple way provided by \textsf{GAP} to give a rough idea of the time spent: \texttt{time} is a global variable that stores the time in milliseconds the last command took.
\begin{example}\label{example:first}
We illustrate how to compute the set of numerical semigroups having $\{19,29\}$ as set of pseudo-Frobenius numbers. Double semicolon in \textsf{GAP} inhibits the output.
\begin{verbatim}
gap> pf := [19,29];;
gap> nsf29 := NumericalSemigroupsWithFrobeniusNumber(29);;time;
31372
gap> Length(nsf29);
34903
gap> nspf1929 := Filtered(nsf29, s -> PseudoFrobeniusOfNumericalSemigroup(s) = pf);;
gap> time;
2540
gap> Set(nspf1929,MinimalGeneratingSystem);                         
[ [ 3, 22, 32 ], [ 6, 9, 16, 26 ], [ 7, 9, 17 ], [ 8, 9, 14 ], [ 8, 9, 15, 22, 28 ], 
[ 9, 12, 13, 14 ], [ 9, 12, 13, 15, 23 ], [ 9, 12, 14, 16, 22 ], 
[ 9, 12, 15, 16, 22, 23, 26 ], [ 9, 13, 14, 17, 21, 24, 25 ], 
[ 9, 13, 15, 17, 21, 23, 25 ], [ 9, 14, 16, 17, 21, 22, 24 ], 
[ 9, 15, 16, 17, 21, 22, 23, 28 ] ]
\end{verbatim}
\end{example}
Prior to the obtention of the procedures that are the object of study of the present paper, we got the necessary insight through detailed analysis of many examples with \PFsf\ consisting of small numbers (less than $29$, say).  Example~\ref{example:first} 
illustrates that this can be easily done by using the \textsf{numericalsgps} package. When \PFsf\ consists of small numbers the time spent is acceptable.

The implementation of our algorithm (available in Version~0.99 of the \textsf{numericalsgps} package) is much faster. It is even faster than just filtering among the numerical semigroups with a given set of pseudo-Frobenius numbers. 
\begin{example}\label{example:second}
To illustrate, we continue the \textsf{GAP} session started in Example~\ref{example:first}.
\begin{verbatim}
gap> new := NumericalSemigroupsWithPseudoFrobeniusNumbers(pf);;time;
29
gap> Set(new,MinimalGeneratingSystem)=Set(nspf1929,MinimalGeneratingSystem);
true
\end{verbatim}
\end{example}

Visualization of images obtained by using the \textsf{intpic}~\cite{intpic} \textsf{GAP} package have also helped us to improve our understanding of the problem and to get the necessary intuition.

The production of Figure~\ref{fig:pfs_19_29} takes less than two minutes in our laptop (using  \textsf{numericalsgps} and \textsf{intpic} packages). Taking into account that there are $34903$ numerical semigroups with Frobenius number $29$, one can consider the performance of the implementations satisfactory.  
In order to explain the meaning of the colors used in this picture to highlight some integers, we anticipate the results shown in Example~\ref{example:forced_picture}. For $\PFsf=\{19,29\}$, we have that $\{1,2,4,5,10,11,19,20\}$ consists of forced gaps and that $\{0, 9, 18, 24, 25, 27, 28, 30\}$ consists of forced elements. To the elements in each of these sets, as well as the ones in the sets of minimal generators, is assigned one color (\emph{red} corresponds to pseudo-Frobenius numbers, \emph{blue} to minimal generators, \emph{green} to elements, \emph{cyan} to forced gaps, and \emph{magenta} to forced elements; in a black and white visualization of this manuscript this will correspond with different gray tonalities). For integers that belong to more that one set, gradient colors are assigned.

\begin{figure}
\begin{center}
\includegraphics[scale=0.58]{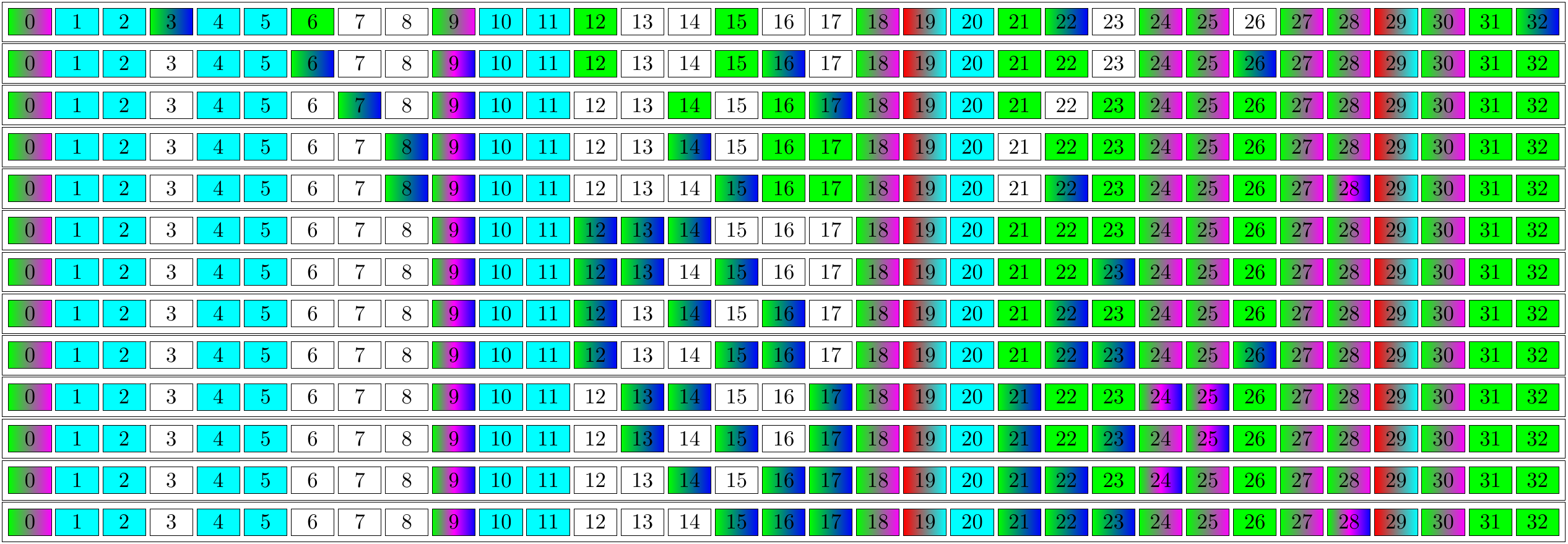}
\end{center}
\caption{The numerical semigroups with pseudo-Frobenius numbers $\{19, 29\}$.\label{fig:pfs_19_29}}
\end{figure}

\section{Integers forced by pseudo-Frobenius numbers - some preliminary procedures}
\label{sec:preliminary_forced_ints}
In this section we give pseudo-code for several functions implementing results given in Section~\ref{sec:generalities} and some others needed later in Section~\ref{sec:computing_forced_ints}.

Once more, $\PFsf$ is a fixed set $\{g_1<g_2<\cdots<g_{n-1}<g_n\}$ of positive integers; \frobsf\ stands for $g_n$ and \typesf\ stands for~$n$.

\subsection{Forced gaps}
\label{subsec:forced_gaps_pseudo_code}
The function~\ref{func:StartingForcedGaps} returns the integers considered in Subsection~\ref{subsec:starting_forced_gaps}, which we called starting forced gaps. Recall (Corollary~\ref{cor:starting_forced_gaps}) that these have to be gaps of all the numerical semigroups having \PFsf\ as set of pseudo-Frobenius numbers.
The justifications for Lines~\ref{line:type_multiplicity} and~\ref{line:pseudo-comb-pseudo} are given by Remark~\ref{rem:type_multiplicity} and Lemma~\ref{lemma:pseudo-comb-pseudo}, respectively.
Line~\ref{line:exclusion} is justified by Lemma~\ref{lemma:maximals}(ii): when it is detected a gap that had to be an element (forced by exclusion), there exists a contradiction and the function returns \fail.
This proves the following proposition.
\begin{proposition}\label{prop:correction_StartingForcedGaps}
Let \textsf{OUT} be the output of the function \ref{func:StartingForcedGaps} for the input \PFsf. If \textsf{OUT} is not \fail, then it consists of gaps of any numerical semigroup $S\in \calSoper(\PFsf)$.
\end{proposition}
\begin{function}[ht]\caption{StartingForcedGaps()\label{func:StartingForcedGaps}}
\StartingForcedGapsw{\PF}\\
\nl\label{line:type_multiplicity}
  $\mathit{fg} := \Union([1..\type],\PF)$; \Comment*[f]{uses the fact $\mathrm{m}(S)\ge\mathrm{t}(S)+1$ (Remark~\ref{rem:type_multiplicity})}\\
  $closures := \List([1..\type-1], i\to \SmallElements(\closure(\{\PF[1],\PF[2],\ldots,\PF[i]\},\PF[i+1])))$\;
  \emph{diffs} := $\emptyset$\;
\nl\label{line:pseudo-comb-pseudo}
  \For{$i \in [2..\type]$}{
    \emph{diffs} $:= \Union(\mathit{diffs},\PositiveInt(\PF[i] - closures[i-1]))$;\Comment*[f]{see Lemma~\ref{lemma:pseudo-comb-pseudo}}}
$\mathit{fg} :=\Divisors(\Union(\mathit{fg},\mathit{diffs}))$\;

\Comment{Detect possible contradiction given by Lemma~\ref{lemma:maximals}(ii)}
\nl\label{line:exclusion}
\For{$x \in \Difference(fg,\PF))$}{\If{$\PositiveInt(\Difference(\PF-x, fg)) = \emptyset$}{\Return \fail}
}
\Return $\mathit{fg}$;
\end{function}
\begin{example}\label{example:starting_forced_gaps}
Let $\PFsf=\{16,29\}$. One can easily check that the set of starting forced gaps is $\{1,2,4,8,13,16,29\}$. The corresponding \textsf{GAP} session:
\begin{verbatim}
gap> pf := [16,29];;
gap> g := StartingForcedGapsForPseudoFrobenius(pf);      
[ 1, 2, 4, 8, 13, 16, 29 ]
\end{verbatim}
\end{example}
The function \ref{func:FurtherForcedGaps} is used to determine forced gaps when some gaps and some elements of a numerical semigroup are known.
The justification for the fact that the output of \ref{func:FurtherForcedGaps} consists of gaps (unless there is an element that also had to be a gap, in which case it returns \emph{fail}) is the following: if $f-e$ and $e$ are elements of a semigroup then $e+f-e=f$ belongs to the semigroup. In particular, if $e$ is an element and $f$ is a gap, then $f-e$ is either negative or a gap. 
This proves the following proposition.
\begin{proposition}\label{prop:correction_FurtherForcedGaps}
Let \textsf{OUT} be the output of the function \ref{func:FurtherForcedGaps} for the input  $(\fg,\fe)$, with \fg and \fe consisting of \PFsf-forced gaps and \PFsf-forced elements, respectively. If \textsf{OUT} is \fail, then $\calSoper(\PFsf)=\emptyset$. Otherwise, \textsf{OUT} consists of gaps of any numerical semigroup $S\in \calSoper(\PFsf)$.
\end{proposition}

\begin{function}[ht]\caption{FurtherForcedGaps()\label{func:FurtherForcedGaps}}
\FurtherForcedGapsw{\fg,\fe}\Comment*[f]{\fg and \fe consist of gaps and elements,}\\
 \Comment*[f]{respectively}\\
  $ng := \Divisors(\PositiveInt(\Union(\List(\fg, g \to g - \fe))))$\;
  \nl\label{line:conflicts_gaps}
  \eIf{$\Intersection(\fe,ng) = \emptyset$}{\Return $\Union(\fg,ng)$\;}{\Return fail\;}
\end{function}  

\subsection{Forced elements}
\label{subsec:forced_elements_pseudo_code}

We use two ways to get new forced elements. One of these ways makes use of Lemma~\ref{lemma:maximals}(ii). We refer to the elements obtained in this way as \emph{elements forced by exclusion}. Another way makes use of the following lemma, which tells us that small gaps force elements that are close to the maximum of \PFsf. Sometimes we refer to them by using the more suggestive terminology \emph{big forced elements}.
\begin{lemma}\label{lemma:big_elts}
Let $m$ be the multiplicity of a numerical semigroup $S$ and let $i$ be an integer such that $1\le i <m $. Then either $\mathrm{F}(S)-i\in S$ or $\mathrm{F}(S)-i\in\PFoper(S)$.
\end{lemma}
\begin{proof}
It suffices to observe that, as $i <m$, one has that  $\mathrm{F}(S)-i +m>\mathrm F(S)$, and consequently $\mathrm{F}(S)-i +m\in S$. The result follows immediately from the definition of pseudo-Frobenius numbers.
\end{proof}

\begin{function}[ht]\caption{FurtherForcedElements()\label{func:FurtherForcedElements}}
\FurtherForcedElementsw{\fg,\fe}\Comment*[f]{\fg and \fe consist of gaps and elements, respectively}\\

\Comment{Big forced elements}
  $m := \First(\Integers,n \to n > 0 \And n \not\in \fg)$;\Comment*[f]{least integer that is not a forced gap}\\
\nl\label{line:big_elts}
  $be := \Difference(\frob - \{1,\ldots,m - 1\}, \PF)$\;

\Comment{Elements forced by exclusion}
 $ee := \emptyset$\;
\nl\label{line:excl_elts_if}
\For{$x\in \fg$}{
  $filt := \Filtered(\PF, f \to \Not((f-x \in \fg) \Or (f-x < 0)))$\;
    \If{$\Length(filt) = 1\ and\ filt[1]-x \not\in ee$}{
       $\AddSet(ee, filt[1]-x)$\;}}
$candidates := \Difference([1 .. \frob-1], \Union(\fg,\Union(\fe,ee)))$\;
\nl\label{line:excl_elts_only_if}
\For{$x \in candidates$}{
    \If{$\PositiveInt(\Difference(\PF-x, \fg)) = \emptyset$}{$\AddSet(ee, x)$\;}}
  $ne := \Union(\fe,ee,be)$\;
  \nl\label{line:conflicts_elts}
  \eIf{$\Intersection(\fg,ne)= \emptyset$}{\Return $ne$\;}{\Return \fail\;}
\end{function}
We observe that in the function \FurtherForcedElements it is used that \PFsf\ is precisely the set of pseudo-Frobenius numbers; otherwise there is no guarantee that the output consists of forced elements.

Let us now prove the correctness of this function. Justification for the the result produced by the cycle starting in Line~\ref{line:excl_elts_if} is given by the direct implication of Lemma~\ref{lemma:maximals}(ii). For the cycle starting in Line~\ref{line:excl_elts_only_if} is given by the reverse implication of the same lemma.
 (The integers known to be forced gaps are assumed to be gaps.)

Justification for Line~\ref{line:big_elts} follows from Lemma~\ref{lemma:big_elts}, since the $m$ appearing there is smaller than or equal to the multiplicity. 

We have then the following proposition.
\begin{proposition}\label{prop:correction_FurtherForcedElements}
Let \textsf{OUT} be the output of the function \ref{func:FurtherForcedElements} for the input  $(\fg,\fe)$, with \fg and \fe consisting of \PFsf-forced gaps and \PFsf-forced elements, respectively. If \textsf{OUT} is ``fail'', then $\calSoper(\PFsf)=\emptyset$. Otherwise, \textsf{OUT} consists of elements of any numerical semigroup $S\in \calSoper(\PFsf)$.
\end{proposition}

\subsection{A condition based on forced integers}
\label{subsec:condition_based_on_forced_integers}
When searching for forced integers, one should pay attention to the existence of possible contradictions.
\begin{example}
Let $PF=\{4,9\}$. Taking divisors and the difference $9-4$, one immediately sees that the set of starting forced gaps contains $\{1, 2, 3, 4, 5, 9\}$. But then $5$ appears as a forced gap and as a (big) forced element. This is a contradiction which shows that $\{4,9\}$ cannot be the set of pseudo-Frobenius numbers of a numerical semigroup. 
\end{example}
 According to Proposition~\ref{prop:intersection_gaps_elements}, the only that we need to take into account is that the set of forced gaps is disjoint from the set of forced elements.

At the end of the functions (Line~\ref{line:conflicts_gaps} in function \ref{func:FurtherForcedGaps}, and Line~\ref{line:conflicts_elts} in function \ref{func:FurtherForcedElements}), possible contradictions are detected.

\section{A procedure to compute integers forced by pseudo-Frobenius numbers}
\label{sec:computing_forced_ints}
Let $\PFsf= \{g_1,g_2,\ldots,g_{n-1},g_n\}$ be a set of positive integers. The aim of this section is to give a procedure to compute some elements of $\calGFoper(\PFsf)$ and $\calEFoper(\PFsf)$.

We start with a subsection that contains a procedure that makes use of the functions given in Section~\ref{sec:preliminary_forced_ints}, which implement some well-known facts recalled in Section~\ref{sec:generalities}.


We then devote a subsection to what we call admissible integers. This will yield another procedure (in Subsection~\ref{subsec:procedure_forced_integers}), which, at the cost of increasing the execution time, may find more forced integers.

The main algorithm of the present paper (Algorithm~\ref{alg:NumericalSemigroupsWithPseudoFrobeniusNumbers}) would work as well by using only the quick version. Our experiments led us to consider the option of using the slower version once, and then use the quick version inside a recursive function that is called by the main algorithm.

\subsection{A quick procedure to compute forced integers}
\label{subsec:quick_forced_integers}

\begin{algorithm}[ht]\caption{SimpleForcedIntegers\label{alg:SimpleForcedIntegers}}
\Input{$\fgs,\fes$, where $\fgs,\fes$ are sets of \PFsf-forced gaps and \PFsf-forced elements, respectively.}
\Output{[\fg, \fe], where
$\fg\supseteq \fgs$ is a set of \PFsf-forced gaps and 
$\fe\supseteq \fes$ is a set of \PFsf-forced elements,
or \fail, when some inconsistency is discovered}
\BlankLine
$\fg := \ShallowCopy(\fgs)$; \Comment*[f]{used to store new forced elements and gaps}\\
$\fe := \ShallowCopy(\fes)$; \Comment*[f]{without creating conflicts in the memory}\\
\Repeat{\Not changes}{
   $changes := \false$\;
\nl\label{line:simple_further_gaps}
   $gaps := \furtherForcedGaps(\fg,\fe)$\;
   \eIf{gaps = \fail}{\Return \fail}{\If{$gaps\ne\fg$}{
     $changes := \true$\;
     $\fg := gaps$;}}
\nl\label{line:simple_further_elts}
   $elts := \furtherForcedElements(\fg,\fe)$\;
   \eIf{elts = \fail}{\Return \fail}{\If{$elts\ne\fe$}{
     $changes := \true$\;
     $\fe := elts$;}}
}
\Return $[\fg,\fe]$;
\end{algorithm}

The correctness of Algorithm~\ref{alg:SimpleForcedIntegers} (\SimpleForcedIntegers) follows from Propositions~\ref{prop:correction_FurtherForcedGaps} and~\ref{prop:correction_FurtherForcedElements} which state the correction of the functions considered in Section~\ref{sec:preliminary_forced_ints}.
\begin{theorem}\label{th:correction_quick_forced_integers}
Algorithm~\ref{alg:SimpleForcedIntegers} correctly produces the claimed output. 
\end{theorem}
Note that the forced elements returned by Algorithm~\ref{alg:SimpleForcedIntegers} (the list in the second component) are obtained by applying the function Closure to some set. Therefore we observe the following.
\begin{remark}\label{rem:forced_elts_are_small_elts_QV}
The second component of the list returned by Algorithm~\ref{alg:SimpleForcedIntegers} is the set of small elements of a numerical semigroup.
\end{remark}
\begin{example}\label{example:all_forced}
Let us see how the \SimpleForcedIntegers works for $\PFsf=\{16,29\}$. We already know (Example~\ref{example:starting_forced_gaps}) that the set of starting forced gaps is $sfg = \{1,2,4,8,13,16,29\}$.

Let us now make a call to \SimpleForcedIntegers with input $sfg, [\ ]$.

The first passage in Line~\ref{line:simple_further_gaps} does not produce any new integer.

The first passage in Line~\ref{line:simple_further_elts} produces 
$$\{0, 3, 6, 9, 12, 15, 18, 21, 24, 25, 27, 28, 30\}$$ 
as current forced elements. Observe that $3$ is forced by exclusion (note that $3=16-13$; also, $29-13=16$ and $16$ is a forced gap); $25$ is also forced by exclusion (note that $16-25<0$ and $29-25=4$ is a forced gap). Also, $21$ is forced by exclusion, but for now we do not need to worry with the multiples of $3$, because these will appear when taking the closure.

The second passage in Line~\ref{line:simple_further_gaps} produces 
$$\{1, 2, 4, 5, 7, 8, 10, 11, 13, 14, 16, 17, 20, 23, 26, 29\}$$ 
as current forced gaps.
To check it, observe that $29-3=26$, $29-6=23$, $29-9=20$, etc. are forced gaps. But then, $10$ and $5$ are (forced to be) gaps.

No further forced elements appear. In fact, the union of the sets of forced gaps and of forced elements is $\{0,\ldots,30\}$.

Therefore, all positive integers less than $29$ are forced. One can check that the closure of the set of forced elements does not produce new forced elements, thus it is the set of small elements of a numerical semigroup. Also, one can check that no forced gap outside $\{16,29\}$ is a pseudo-Frobenius number, thus one may conclude that there exists exactly one numerical semigroup $S$ such that $\PFoper(S)=\{16,29\}$.

This example illustrates that more than one passage through the \emph{repeat-until} loop of the algorithm \SimpleForcedIntegers may be needed.
\end{example}
\begin{example}\label{example:forced_picture}
In this example we present a \textsf{GAP} session, now for $\PFsf\{19,29\}$.
\begin{verbatim}
gap> pf := [19,29];;                            
gap> sfg := StartingForcedGapsForPseudoFrobenius(pf);
[ 1, 2, 5, 10, 19, 29 ]
gap> SimpleForcedIntegersForPseudoFrobenius(sfg,[],pf);
[ [ 1, 2, 4, 5, 10, 11, 19, 20, 29 ], [ 0, 9, 18, 24, 25, 27, 28, 30 ] ]
\end{verbatim}
Recall that the names used in our package are longer than the ones used in this manuscript, so for instance, \texttt{SimpleForcedIntegersForPseudoFrobenius} is the name we have used in our package for Algorithm~\ref{alg:SimpleForcedIntegers}. 
This example is related to Figure~\ref{fig:pfs_19_29}, where forced integers are highlighted.
\end{example}
\begin{example}\label{ex:forced_ints_for_tree}
Let us now apply the algorithm to $\PFsf=\{15, 20, 27, 35\}$. Again, we will use \textsf{GAP} to help us in doing the calculations (which can be easily confirmed by hand).
\begin{verbatim}
gap> pf := [ 15, 20, 27, 35 ];;
gap> sfg := StartingForcedGapsForPseudoFrobenius(pf);                    
[ 1, 2, 3, 4, 5, 6, 7, 8, 9, 10, 12, 15, 20, 27, 35 ]
\end{verbatim}
We immediately get that $\{ 25, 26, 28, 29, 30, 31, 32, 33, 34\}$ consists of forced big elements. And one can observe that $19$ and $23$ are forced by exclusion. This leads to the obtention of $35-19 =16$ as forced gap.
No other forced elements are obtained, which agrees with the following continuation of the \textsf{GAP} session:
\begin{verbatim}
gap> SimpleForcedIntegersForPseudoFrobenius(sfg,[],pf);
[ [ 1, 2, 3, 4, 5, 6, 7, 8, 9, 10, 12, 15, 16, 20, 27, 35 ], 
  [ 0, 19, 23, 25, 26, 28, 29, 30, 31, 32, 33, 34, 36 ] ]
\end{verbatim}
\end{example}

\subsection{Admissible integers}
\label{subsec:admissible_integers}
Let $G$ and $E$ respectively be sets of \PFsf-forced gaps and \PFsf-forced elements, and let $v$ be a free integer for $(G,E)$. We say that $v$ is \emph{admissible for} $(G,E)$ if Algorithm~\ref{alg:SimpleForcedIntegers} when applied to $(G,E\cup\{v\})$ does not return \emph{fail}. Otherwise, we say that $v$ is \emph{non-admissible for} $(G,E)$. Thus, $v$ is non-admissible implies that $v$ can not be an element of any semigroup in $\calSoper(\PFsf)$ and therefore is a gap of all semigroups in $\calSoper(\PFsf)$, that is, is a forced gap.

\begin{lemma}\label{lemma:admissible_ints}
Let  $G$ and $E$ be sets of forced gaps and forced elements, respectively. Let $v$ be free for $(G,E)$. If $v$ is non-admissible for $(G,E)$, then $v$ is a \PFsf-forced gap.
\end{lemma}
Observe that a semigroup generated by admissible elements for some pair $(G,E)$ consists of admissible elements for $(G,E)$.

The function \ref{func:NonAdmissible}, with input a pair of sets of forced gaps and forced elements, returns non-admissible integers, which, by Lemma~\ref{lemma:admissible_ints}, are new forced gaps. This function is called by Algorithm~\ref{alg:ForcedIntegers}, a not so quick procedure to compute forced integers.
\begin{proposition}\label{prop:correction_NonAdmissible}
Let \textsf{OUT} be the output of the function \ref{func:NonAdmissible} for the input  $(\fg,\fe)$, with \fg and \fe consisting of \PFsf-forced gaps and \PFsf-forced elements, respectively. Then \textsf{OUT} is a set of integers that are non-admissible for $(\fg,\fe)$.
\end{proposition}
\begin{function}[ht]\caption{NonAdmissible()\label{func:NonAdmissible}}
\NonAdmissiblew{\fg,\fe}\Comment*[f]{\fg and \fe consist of forced gaps and forced elements, respectively}\\
    $admissible := \emptyset$\;
    $totest := \Difference([1..\frob], \Union(\fg,\fe))$;\Comment*[f]{start with the free integers}\\
    \While{$totest \ne\emptyset$}{
      $v := totest[1]$\;
      $pnfce:= \SimpleForcedIntegers(\fg,\Union(\fe,[v]),\PF)$\;
      \eIf{$pnfce \ne \fail$}{
        $admissible := \Union(admissible,pnfce[2])$\;
        $totest := \Difference(totest,admissible)$\;}{
        $totest := Difference(totest,[v])$;}
    }
  \Return $\Difference([1..\frob], admissible)$;
\end{function}
\begin{example}\label{example:non_admissible}
Let $PF=\{11, 22, 23, 25\}$. 
\begin{verbatim}
gap> pf := [ 11, 22, 23, 25 ];;
gap> sfg := StartingForcedGapsForPseudoFrobenius(pf);                    
[ 1, 2, 3, 4, 5, 6, 7, 11, 12, 14, 22, 23, 25 ]
\end{verbatim}
By using the function \SimpleForcedIntegers one obtains the following.
\begin{verbatim}
gap> SimpleForcedIntegersForPseudoFrobenius(sfg,[],pf);
[ [ 1, 2, 3, 4, 5, 6, 7, 11, 12, 14, 22, 23, 25 ], 
  [ 0, 18, 19, 20, 21, 24, 26 ] ]
\end{verbatim}
That $\{1, 2, 3, 4, 5, 6, 7, 11, 12, 14, 22, 23, 25\}$ consists of forced gaps and that the set $\{0, 18, 19, 20, 21, 24, 26\}$ can easily be confirmed by hand.

Let us now check that $15$ is non-admissible. If it was an element of a semigroup $S\in\calSoper(PF)$, then $10(=25-15)$ and $8(=23-15)$ would be gaps of $S$. But then $17$ is a big element, $13 (= 23-10)$ is forced by exclusion (note that $25-10=15$, $22-10=12$ and $11-10=1$ are gaps) and $9(=23-14)$ is forced by exclusion too ($25-14=11$, $22-14=8$ are gaps and $11-14<0$). This is not possible, since $13+9=22$ is a gap. Therefore $15$ is non-admissible.
\end{example}
\subsection{A not so quick procedure to compute forced integers}
\label{subsec:procedure_forced_integers}
Algorithm~\ref{alg:ForcedIntegers} is our procedure to compute forced integers that produces the best result in terms of the number of forced integers encountered. Besides Algorithm~\ref{alg:SimpleForcedIntegers}, it makes use of the function \ref{func:NonAdmissible}.
\begin{remark}\label{rem:forced_elts_are_small_elts}
It is a consequence of Remark~\ref{rem:forced_elts_are_small_elts_QV} that the second component of the list returned by Algorithm~\ref{alg:ForcedIntegers} is the set of small elements of a numerical semigroup.
\end{remark}

\begin{algorithm}[ht]\caption{ForcedIntegers\label{alg:ForcedIntegers}}
\Input{\PF}
\Output{\fail if some inconsistency is discovered; otherwise, returns [\fg, \fe], where \fg and \fe are sets of forced gaps and forced elements, respectively.}

  \If{\type = 1}{\Return [Divisors(\frob),[0,\frob+1]];}
    $fints := \SimpleForcedIntegers(\startingForcedGaps(\PF),[\ ])$\;
    \eIf{fints = \fail}{\Return \fail}{\If{\IsRange$(\Union(fints))$}{\Return fints;}}
    $nad := \nonAdmissible(fints[1],fints[2])$\;   
    $newgaps := \Difference(nad,fints[1])$\;
    \Return $\SimpleForcedIntegers(\Union(newgaps,fints[1]),fints[2])$;
\end{algorithm}
The correctness of Algorithm~\ref{alg:ForcedIntegers} follows from Proposition~\ref{prop:correction_NonAdmissible} and Theorem~\ref{th:correction_quick_forced_integers}.
\begin{theorem}\label{th:correction_forced_integers}
Algorithm~\ref{alg:ForcedIntegers} correctly produces the claimed output.
\end{theorem}

\subsection{Examples and execution times}
\label{subsec:examples_and_execution_times}
Algorithm~\ref{alg:SimpleForcedIntegers} can be used as a quick way to compute forced integers. In fact, when called with the starting forced gaps and the empty set of elements, can be seen as a quick version of Algorithm~\ref{alg:ForcedIntegers}. We use \begin{center}
\texttt{ForcedIntegers\_QV(\PFsf)}\end{center} 
 as a shorthand for \begin{center}
\texttt{SimpleForcedIntegers(\startingForcedGaps(\PFsf),[\ ])}
\end{center} 

In the following example we use the names for the functions in the current implementation.
\begin{example}
This example is meant to illustrate the difference between applying the quick and the normal version of the algorithm. 
\begin{verbatim}
gap> pf := [ 103, 110, 112, 137, 160, 178, 185 ];;              
gap> fQ := ForcedIntegersForPseudoFrobenius_QV(pf);;time;       
25
gap> fN := ForcedIntegersForPseudoFrobenius(pf);;time;          
262
gap> Length(fQ[1]);Length(fN[1]);Length(fQ[2]);Length(fN[2]);   
50
51
29
31
\end{verbatim}
We have used the internal names in the package \texttt{ForcedIntegersForPseudoFrobenius\_QV} and \texttt{Forced\-Integers\-For\-PseudoFrobenius} for \ForcedIntegersQV and \ForcedIntegers, respectively.

\end{example}
Table~\ref{fig:execution_time_forced_integers} collects some information concerning some execution times (as given by \textsf{GAP}) and the number of forced gaps and of forced elements both using Algorithm~\ref{alg:SimpleForcedIntegers} (identified as QV (which stands for \emph{quick version})) and Algorithm~\ref{alg:ForcedIntegers} (identified as NV (which stands for \emph{normal version})).
We observe that the execution times when using the quick version remain relatively small, even when the Frobenius number is large.


\begin{table}
\begin{center}
    \begin{tabular}{| l | l | l | l | l | l | l |}
    \hline
\multirow{2}{*}{pseudo Frobenius numbers}&\multicolumn{2}{|c|}{Time}&\multicolumn{2}{|c|}{\# F. Gaps}&\multicolumn{2}{|c|}{\# F. Elements}\\ \cline{2-7}
                          &QV & NV            & QV & NV                     &QV & NV           \\ \hline
[ 11, 22, 23, 25 ]&2&7&13&14&7&7\\ \hline
[ 17, 27, 28, 29 ]   &   2&5&16&17&8&10\\ \hline
[ 17, 19, 21, 25, 27 ]&2&9&15&16&8&8\\ \hline
[ 15, 20, 27, 35 ]&2&10&16&16&13&13\\ \hline
[ 12, 24, 25, 26, 28, 29 ]&3&11&18&22&6&9\\ \hline
[ 145, 154, 205, 322, 376, 380 ]&47&1336&82&85&52&54\\ \hline
[ 245, 281, 282, 292, 334, 373, 393, 424, 432, 454, 467 ]&129&2075&116&116&53&53\\ \hline
[ 223, 434, 476, 513, 549, 728, 828, 838, 849, 953 ]&213&5866&300&318&221&253\\ \hline
[ 219, 437, 600, 638, 683, 779, 801, 819, 880 ]&205&4838&219&224&153&161\\ \hline
[ 103, 110, 112, 137, 160, 178, 185 ]&25&262&50&51&29&31\\ \hline
    \end{tabular}
\end{center}
\caption{Execution times for computing forced integers.\label{fig:execution_time_forced_integers}}
\end{table}

Failure is usually detected very quickly, as should be clear and Table~\ref{fig:execution_time_failure} somehow confirms.


\begin{table}
\begin{center}
    \begin{tabular}{| l | l | l |}
    \hline
\multirow{2}{*}{pseudo Frobenius numbers}&\multicolumn{2}{|c|}{Time}\\ \cline{2-3}
                          &QV & NV\\ \hline
[ 18, 42, 58, 88, 94 ]&6&6\\ \hline
[ 20, 27, 34, 35, 37, 42, 48, 80 ]&2&3\\ \hline
[ 30, 104, 118, 147, 197, 292, 298, 315, 333, 384, 408 ]&32&43\\ \hline
[ 36, 37, 219, 233, 304, 410, 413, 431, 438, 458 ]&35&32\\ \hline
[ 89, 411, 446, 502, 557, 600, 605, 631, 636, 796, 801, 915 ]&223&233\\ \hline
[ 56, 134, 136, 137, 158, 248, 277, 373, 383, 389, 487, 558, 566, 621, 691, 825, 836 ]&103&113\\ \hline
    \end{tabular}
\end{center}
\caption{When failure occurs...\label{fig:execution_time_failure}}
\end{table}
As one could expect, there are examples where failure is not detected with the quick version.
\begin{verbatim}
gap> pf := [ 25, 29, 33, 35, 38, 41, 46 ];
[ 25, 29, 33, 35, 38, 41, 46 ]
gap> ForcedIntegersForPseudoFrobenius(pf);
fail
gap> ForcedIntegersForPseudoFrobenius_QV(pf);
[ [ 1, 2, 3, 4, 5, 6, 7, 8, 9, 10, 11, 12, 13, 16, 17, 19, 21, 23, 25, 29, 33, 35, 
 38, 41, 46 ], [ 0, 30, 34, 36, 37, 39, 40, 42, 43, 44, 45, 47 ] ]
gap> pf := [ 22, 23, 24, 25, 26 ];
[ 22, 23, 24, 25, 26 ]
gap> ForcedIntegersForPseudoFrobenius_QV(pf);
[ [ 1, 2, 3, 4, 5, 6, 8, 11, 12, 13, 22, 23, 24, 25, 26 ], [ 0, 20, 21, 27 ] ]
gap> ForcedIntegersForPseudoFrobenius(pf);
fail
\end{verbatim}

On the other hand, we have not been able to detect any set \PFsf\ candidate to be the set of pseudo-Frobenius numbers of a numerical semigroup that passes the normal version and such that $\calSoper(\PFsf)=\emptyset$. Despite the various millions of tests made, we do not feel comfortable on leaving it as a conjecture; we leave it as a question instead.
\begin{question}\label{quest:free_implies_nonempty}
If Algorithm~\ref{alg:ForcedIntegers} with input \PFsf\ does not return \emph{fail}, then $\calSoper(\PFsf)\ne\emptyset$?
\end{question}
We observe that in view of the execution times illustrated in Table~\ref{fig:execution_time_forced_integers}, a positive answer to Question~\ref{quest:free_implies_nonempty} would imply that Algorithm~\ref{alg:ForcedIntegers} constitutes a satisfactory answer to Question~\ref{quest:existence}. 

\section{An approach based on forced integers}
\label{sec:approach_based_on_forced_integers}
In this section we present our main algorithm (Algorithm~\ref{alg:NumericalSemigroupsWithPseudoFrobeniusNumbers}), which computes $\calSoper(\PFsf)$. Its correctness is stated in Theorem~\ref{th:correction_of_algorithm}, whose proof is built from almost all the results preceding it in the paper.

After considering some initial cases, the algorithm makes a call to \ref{func:RecursiveDepthFirstExploreTree} which is a recursive function used to construct a tree whose nodes are labeled by pairs $(X,Y)$ where $X$ is a list of forced gaps and $Y$ is a list of forced elements. Thus we implicitly have lists of free integers in each node: the complement of $X\cup Y$ in the set $U=\{1,\ldots, g_{n}\}$, where $\PFsf=\{g_1<\cdots < g_n\}$. Nodes with an empty set of free integers are the leafs in our tree. 

A node $(X,Y)$ such that there exists a numerical semigroup $S\in\calSoper(\PFsf)$ for which $X\subseteq \Goper(S)$ and $Y\subseteq \SEoper(S)$ is said to be \PFsf-\emph{admissible}, or simply \emph{admissible}, if \PFsf\ is understood. A node that is not admissible is called a \emph{dead node}. 

\begin{remark}
The knowledge of some forced integers allows us to identify immediately some dead nodes: if $(G,E)$ is such that $G$ consists of forced gaps and $E$ consists of forced elements, then any node $(X,Y)$ such that $X\cap E\ne \emptyset$ or $Y\cap G\ne \emptyset$ is a dead node.  
\end{remark}

\begin{remark}\label{rem:leaf_semigroup}
Let $(X,Y)$ be a leaf that is not a dead node. It follows from the construction (see Remarks~\ref{rem:forced_elts_are_small_elts_QV} and~\ref{rem:forced_elts_are_small_elts}) that there exists a numerical semigroup $S$ such that $(\Goper(S), \SEoper(S)) = (X,Y)$.  
\end{remark}

\begin{remark}\label{rem:call_irr}
	Observe that if $\PFsf=\{g_1\}$ or $\PFsf=\{g_1/2<g_1\}$, then the set of numerical semigroups with pseudo-Frobenius numbers $\PFsf$ corresponds with the set of irreducible numerical semigroups having Frobenius number $g_1$; see Appendix~\ref{sec:computing_arbitrary_type}. In this case we will use the fast procedure presented in \cite{br}.
\end{remark}	
\subsection{The recursive construction of a tree}
\label{subsec:recursive_construction_tree}
A naive idea is to start with a list of free integers (for some set of forced integers) and turn each one of these free integers into either a gap or an element. Assuming that the number of free integers is $n$, the number of possibilities is $2^n$, thus checking each of these possibilities for being in correspondence with the set of gaps of a numerical semigroup with \PFsf\ as set of pseudo-Frobenius numbers is unfeasible, unless $n$ is small. Nevertheless, this naive idea can be adapted and take advantage of the already observed facts that elements can force gaps and vice-versa. Although there are examples for which fixing one integer does not force anything else, which let us expect that nothing good can result from a worst case complexity analysis, in practice it works quite well. We give some examples, but leave a detailed analysis of the  complexity (perhaps average complexity) as an open problem.

The procedure we use is, for each integer $v$ in the current list of free integers, compute all numerical semigroups containing $v$ and the forced elements, and having the forced gaps as gaps. We proceed recursively and  use backtracking when a semigroup 
or a dead node is found.  When we revisit the node, we then suppose that $v$ is a gap and continue with the next free integer. 

Before proceeding with the pseudo-code for the recursive function \ref{func:RecursiveDepthFirstExploreTree} that constructs the tree in a depth first manner, let us give an example which we accompany by a picture (Figure~\ref{fig:tree}).
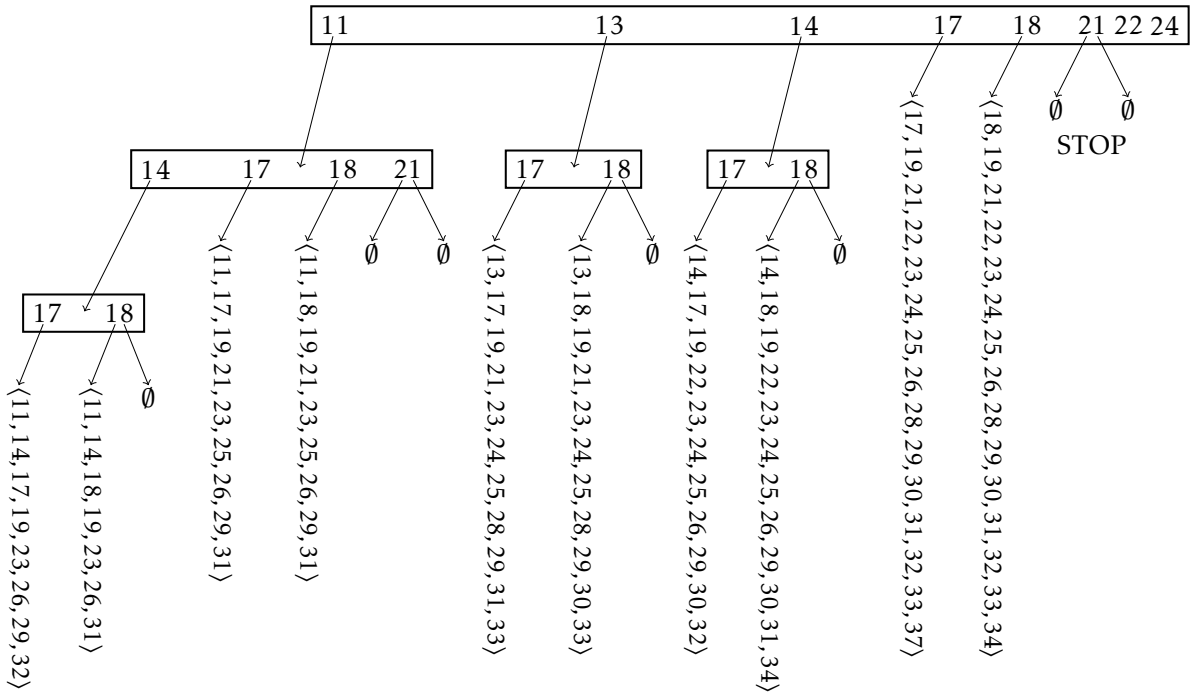
\begin{figure}[ht]
    \centering
    \resizebox{0.9\textwidth}{!}{%
\begin{tikzpicture}[inner sep = 0.5pt]
%
  \node (a11) at (3,10) {11};
  \node (a13) at (6.8,10) {13};
  \node (a14) at (9.5,10) {14};
  \node (a17) at (11.5,10) {17};
  \node (a18) at (12.6,10) {18};
  \node (a21) at (13.5,10) {21};
  \node (a22) at (14,10) {22};
  \node (a24) at (14.5,10) {24};
  \node[inner sep = 3pt, thick,draw, fit=(a11) (a13) (a14) (a17) (a18) (a21) (a22) (a24)] {};

  \node (a11b14) at (0.5,8) {14};
  \node (a11b17) at (1.9,8) {17};
  \node (l21middle) at (2.5,8) {};
  \node (a11b18) at (3.1,8) {18};
  \node (a11b21) at (4,8) {21};
  \path[->] (a11) edge (l21middle);
  \node[inner sep = 3pt, thick,draw, fit=(a11b14) (a11b17) (a11b18) (a11b21)] {};

  \node (a13b17) at (5.7,8) {17};
  \node (l22middle) at (6.3,8) {};
  \node (a13b18) at (6.9,8) {18};
  \path[->] (a13) edge (l22middle);
  \node[inner sep = 3pt,thick,draw, fit=(a13b17) (a13b18)] {};

  \node (a14b17) at (8.5,8) {17};
  \node (l23middle) at (9,8) {};
  \node (a14b18) at (9.5,8) {18};
  \path[->] (a14) edge (l23middle);
  \node[inner sep = 3pt,thick,draw, fit=(a14b17) (a14b18)] {};

  \node (a11b14c17) at (-1,6) {17};
  \node (l3middle) at (-0.5,6) {};
  \node (a11b14c18) at (0,6) {18};
  \path[->] (a11b14) edge (l3middle);
  \node[inner sep = 3pt,thick,draw, fit=(a11b14c17) (a11b14c18)] {};

 \node (sgp1middle) at (-1.4,5) [anchor=north] {\rotatebox{-90}{\small $\langle 11, 14, 17, 19, 23, 26, 29, 32 \rangle$}};
  \path[->] (a11b14c17) edge (-1.4,5);
  %

  \node (sgp2middle) at (-.4,5)  [anchor=north] {\rotatebox{-90}{\small $\langle 11, 14, 18, 19, 23, 26, 31 \rangle$}};
  \path[->] (a11b14c18) edge (-.4,5);
  %

  \node (d1) at (0.4,5) [anchor=north] {$\emptyset$};
  \path[->] (a11b14c18) edge (0.4,5);

  \node (sgp3middle) at (1.4,7) [anchor=north] {\rotatebox{-90}{\small $\langle 11, 17, 19, 21, 23, 25, 26, 29, 31 \rangle$}};
  \path[->] (a11b17) edge (1.4,7);
  %

  \node (sgp4middle) at (2.6,7) [anchor=north] {\rotatebox{-90}{\small $\langle 11, 18, 19, 21, 23, 25, 26, 29, 31 \rangle$}};
  \path[->] (a11b18) edge (2.6,7);
  %

  \node (d2) at (3.5,7) [anchor=north] {$\emptyset$};
  \path[->] (a11b21) edge (3.5,7);
  \node (d3) at (4.5,7) [anchor=north] {$\emptyset$};
  \path[->] (a11b21) edge (4.5,7);

  \node (sgp5middle) at (5.2,7) [anchor=north] {\rotatebox{-90}{\small $\langle 13, 17, 19, 21, 23, 24, 25, 28, 29, 31, 33 \rangle$}};
  \path[->] (a13b17) edge (5.2,7);
  %

  \node (sgp6middle) at (6.4,7) [anchor=north] {\rotatebox{-90}{\small $\langle 13, 18, 19, 21, 23, 24, 25, 28, 29, 30, 33 \rangle$}};
  \path[->] (a13b18) edge (6.4,7);
  %

  \node (d4) at (7.4,7) [anchor=north] {$\emptyset$};
  \path[->] (a13b18) edge (7.4,7);

   \node (sgp7middle) at (8,7) [anchor=north] {\rotatebox{-90}{\small $\langle 14, 17, 19, 22, 23, 24, 25, 26, 29, 30, 32 \rangle$}};
  \path[->] (a14b17) edge (8,7);
  %

  \node (sgp8middle) at (9,7)  [anchor=north]  {\rotatebox{-90}{\small $\langle 14, 18, 19, 22, 23, 24, 25, 26, 29, 30, 31, 34 \rangle$}};
  \path[->] (a14b18) edge (9,7);
  %

  \node (d5) at (10,7) [anchor=north] {$\emptyset$};
  \path[->] (a14b18) edge (10,7);

  \node (sgp9middle) at (11,9)  [anchor=north] {\rotatebox{-90}{\small $\langle 17, 19, 21, 22, 23, 24, 25, 26, 28, 29, 30, 31, 32, 33, 37 \rangle$}};
  \path[->] (a17) edge (11,9);
  %

  \node (sgp10middle) at (12.1,9)  [anchor=north] {\rotatebox{-90}{\small $\langle 18, 19, 21, 22, 23, 24, 25, 26, 28, 29, 30, 31, 32, 33, 34 \rangle$}};
  \path[->] (a18) edge (12.1,9);
  %

  \node (d6) at (13,9) [anchor=north] {$\emptyset$};
  \path[->] (a21) edge (13,9);
  \node (d7) at (14,9) [anchor = north] {$\emptyset$};
  \path[->] (a21) edge (14,9);
  \node (STOP) at (13.5,8.5) [thick, anchor=north] {STOP};

\end{tikzpicture}
}
  \caption{The numerical semigroups with pseudo-Frobenius numbers $\{15, 20, 27, 35\}$.\label{fig:tree}}
\end{figure}

\begin{example}\label{ex:tree}
Let $\PFsf = \{15, 20, 27, 35\}$. From Example~\ref{ex:forced_ints_for_tree} we have a pair of lists of forced gaps and forced integers, which leaves the list 
\begin{equation}\label{eq:free_list}
F=\{11, 13, 14, 17, 18, 21, 22, 24\}
\end{equation} 
of free integers.

The leaves contained in the branch that descends from $11$ in Figure~\ref{fig:tree} consist of the semigroups containing $11$ as forced element.

All the remaining semigroups in $\calSoper(\PFsf)$ must have $11$ as gap. The process then  continues as Figure~\ref{fig:tree} illustrates: considering $13$ as a forced integer, developing its corresponding subtree and so on.

Let us now look at the integer $21$ in the root of the tree. At this point, all the semigroups in $\calSoper(\PFsf)$ containing some integer preceding $21$ in $F$ have been computed. Thus, any new semigroup must have the elements in $\{11, 13, 14, 17, 18\}$ as gaps. From the computations one concludes that no new semigroup appears having $21$ as an element. Thus, if some not already computed semigroup exists fulfilling the current conditions, then it must have $21$ as gap. One can check that this can not happen. The remaining elements ($22$ and $24$) need not to be checked, since $21$ had to be either an element or a gap. Therefore we can stop.
\end{example}

Next we give pseudo-code for \ref{func:RecursiveDepthFirstExploreTree}. The recursion ends when no more than one free integer is left. A call to the function \ref{func:EndingCondition} is then made.

The variable \textbf{semigroups} works as a container which serves to store all the semigroups we are looking for, as they are discovered. 
\begin{function}[ht]\caption{EndingCondition()\label{func:EndingCondition}}
\EndingConditionw{\fgs,\fes}\Comment*[f]{\fgs and \fes are such that $\#U\setminus (\fgs\cup\fes) \le 1$}\\
\Comment*[f]{\fgs and \fes represent lists of gaps and elements, respectively}\\
    $free := \ShallowCopy(\Difference([1..\frob],\Union(\fgs,\fes)))$\;
\nl\label{line:range}     
  \If{$\Length(free) = 0$}{
 \nl\label{line:ending_condition}
\If{$\First(\Difference(\fgs,\PF), pf \to \Intersection(pf + \Difference(\fes,[0]),\fgs) = \emptyset) = \fail$}{$\Add(\semigroups, \NumericalSemigroupByGaps(\fgs))$;} \Return;}

\nl\label{line:one_free}     
    \If{$\Length(free) = 1$}{\If{$\RepresentsGapsOfNumericalSemigroup(g)$}{
 \nl\label{line:ending_condition_left}
\If{$\First(\Difference(g,\PF), pf \to \Intersection(pf + \Difference(\Union(e,free),[0]),g) = \emptyset) = \fail$}{$\Add(\semigroups, \NumericalSemigroupByGaps(g))$;}}
      \If{$\RepresentsGapsOfNumericalSemigroup(\Union(g,free))$}{
\nl\label{line:ending_condition_right}
\If{$\First(\Difference(g,\PF), pf \to \Intersection(pf + \Difference(e,[0]),Union(g,free)) = \emptyset) = \fail$}{$\Add(\semigroups, \NumericalSemigroupByGaps(\Union(g,free)))$;}}
    \Return;
    }
\end{function}
\begin{lemma}\label{lemma:ending_condition}
Function~\ref{func:EndingCondition} either does nothing or adds to \textbf{semigroups} a numerical semigroup $S$ such that $\PFoper(S)=\PFsf$.
\end{lemma}
\begin{proof}
It suffices to observe that any of the conditions in the ``if'' statements of Lines~\ref{line:ending_condition},~\ref{line:ending_condition_left},~\ref{line:ending_condition_right} guarantee that no forced gap outside \PFsf\ can be a pseudo-Frobenius number.
\end{proof}
Notice that when \ref{func:EndingCondition} does nothing, it is because  one of the following reasons:
\begin{itemize}
\item there is no numerical semigroup whose set of gaps is the first component of the input,
\item the resulting semigroup does not have \PFsf\ as set of pseudo-Frobenius numbers (it actually has more pseudo-Frobenius numbers), 
\item there is a free element that cannot be neither a gap nor an element.
\end{itemize}
\begin{function}[ht]\caption{RecursiveDepthFirstExploreTree()\label{func:RecursiveDepthFirstExploreTree}}
\RecursiveDepthFirstExploreTreew{$[\fg,\fe]$}\Comment*[f]{\fg and \fe consist of forced gaps and}\\
\Comment*[f]{ forced elements, respectively}\\
    $currentfree := \ShallowCopy(\Difference([1..\frob],\Union(\fg,\fe)))$\;

    $nfg := \ShallowCopy(\fg)$; \Comment*[f]{used to store new forced gaps}\\
\Comment*[f]{without creating conflicts in the memory}\\
    \While{$\Length(currentfree) > 1$}{
      $v := currentfree[1]$\;
\nl\label{line:left}      $left := \SimpleForcedIntegers(nfg,\Union(\fe,[v]))$\;
\nl\label{line:recursion}      \eIf{$left = \fail$}{
        $right := \SimpleForcedIntegers(\Union(nfg,[v]),\fe)$\;
        \If{$(right = \fail) \Or (\Intersection(right[1],right[2]) \ne \emptyset)$}{\xbreak;}
        }{
      \recursiveDepthFirstExploreTree($[left[1],left[2]]$);
      }
      $nfg := \Union(nfg,[v])$\;
      $currentfree := \Difference(currentfree,[v])$\;
   }
\nl\label{line:single_free}     
  \If{$\Length(currentfree) \le 1$}{\endingCondition(\fg,\fe);}
\end{function}
Observe that recursion in \ref{func:RecursiveDepthFirstExploreTree} ends (in Line~\ref{line:one_free}) when there is at most one free element. 
\begin{proposition}\label{prop:correction_recursive_function}
Let $(fg,fe)$ be a pair of disjoint sets of integers contained in $U=\{0,\ldots,g_n+1\}$. 
After the execution of the function \ref{func:RecursiveDepthFirstExploreTree} with input $(fg,fe)$, 
\textbf{semigroups} contains all numerical semigroups $S$ such that  $S\in \calSoper(\PFsf)$, $fg\subseteq \Goper(S)$ and $fe\subseteq \SEoper(S)$.
\end{proposition}
\begin{proof}
Denote by $\Lambda$ the set of all numerical semigroups $S$ such that  $fg\subseteq \Goper(S)$, $fe\subseteq \SEoper(S)$. We have to prove that $\mathbf{semigroups}\cap \Lambda = \calSoper(\PFsf)\cap \Lambda$.

\noindent $\subseteq$. It suffices to observe that the numerical semigroups are added to \textbf{semigroups} by the function \ref{func:EndingCondition} and these belong to $\calSoper(\PFsf)$, by Lemma~\ref{lemma:ending_condition}.

\noindent $\supseteq$. 
Let $S\in \calSoper(\PFsf)$ be such that $fg\subseteq \Goper(S)$ and $fe\subseteq \SEoper(S)$. 
If $\# U\setminus(fg\cup fe)\in\{0,1\}$, then the function \ref{func:EndingCondition} is called and it gives the correct output. Otherwise, we enter a recursion.

We will prove by induction on the number of free elements that the output is correct.
Let us consider the smallest integer $v\in U\setminus(fg\cup fe)$. \ref{func:RecursiveDepthFirstExploreTree} is called with input $(fg,fe)$, and will enter the while loop, adding $v$ to $fe$ and computing new forced integers (\SimpleForcedIntegers). 
\begin{itemize}
\item If $v\in S$, then $left$ in Line~\ref{line:left} will not be equal to \fail (Theorem \ref{th:correction_quick_forced_integers}), and we will call \ref{func:RecursiveDepthFirstExploreTree} with a lager set $fe$, having in consequence less  free integers. By induction, $S$ is added to \textbf{semigroups}. 
\item Now assume that $v\not\in S$. After the execution of  the if-then-else starting in Line~\ref{line:recursion}, $v$ is added to the set of gaps. We have then one element less in the list of free integers, $fg\cup\{v\}\subseteq \Goper(S)$ and $fe\subseteq \SEoper(S)$, whence the effect is the same as if we called \ref{func:RecursiveDepthFirstExploreTree}  with arguments $fg\cup\{v\}\subseteq \Goper(S)$ and $fe\subseteq \SEoper(S)$. By induction hypothesis, $S$ is added to \textbf{semigroups}.\qedhere
\end{itemize}
\end{proof}
Observe that if $\calSoper(\PFsf)\ne\emptyset$ and the sets $\mathsf{fg}$ and $\mathsf{fe}$ considered in Proposition~\ref{prop:correction_recursive_function} consist of forced gaps and forced elements, respectively, then $\Lambda\subseteq\calSoper(\PFsf)$. Therefore we have proved the following corollary.
\begin{corollary}\label{cor:recursive_function}
If the function \ref{func:RecursiveDepthFirstExploreTree} is called with a pair $({fg,fe})$, where ${fg}$ consists of forced gaps and ${fe}$ consists of forced elements, then at the end of its execution we have that $\mathbf{semigroups}=\calSoper(\PFsf)$.
\end{corollary}
\subsection{The main algorithm}
\label{subsec:the_main_algorithm}
\begin{algorithm}[ht]\caption{NumericalSemigroupsWithPseudoFrobeniusNumbers\label{alg:NumericalSemigroupsWithPseudoFrobeniusNumbers}}
\Input{\PF}
\Output{the set $\calSoper(\PFsf)$.}
\BlankLine
\nl\label{line:type1or2} 
\If{$(\type = 1) \Or  (\type = 2 \And \PF[1]=\PF[2]/2)$}{\Return \textrm{IrreducibleNumericalSemigroupsWithFrobeniusNumber}(\frob);}
\nl\label{line:basic_condition} 
\If{$\Not\ (\PF[\type] - \PF[\type-1] > \PF[1])$}{\Return $\emptyset$;}

\nl\label{line:forced_fail} 
$root$:=\ForcedIntegers(\PF)\;

\If{$root = \fail$}{\Return $\emptyset$;}
$\semigroups:=\emptyset$\;
$\recursiveDepthFirstExploreTree(root)$\;
\Return \semigroups\;
\end{algorithm}

We observe that Algorithm~\ref{alg:NumericalSemigroupsWithPseudoFrobeniusNumbers} is not efficient when there are many free integers (for some set of forced integers). A possible attempt to improve it could be to replace the recursive function by a much more efficient depth first search algorithm to parse the tree in question. Another one is to use other sophisticated theoretical tools. What is done in Appendix~\ref{sec:computing_arbitrary_type} could be seen as an attempt. Having at our disposal more than one approach, we can take advantage of choosing the most efficient for each situation. 
\begin{theorem}\label{th:correction_of_algorithm}
Let $\PFsf=\{g_1<\dots<g_n\}$ be a set of positive integers.
The output of Algorithm~\ref{alg:NumericalSemigroupsWithPseudoFrobeniusNumbers} with input \PFsf\ is $\calSoper(\PFsf)$.
\end{theorem}
\begin{proof}
Line~\ref{line:type1or2} is justified by Remark~\ref{rem:call_irr}. 
The justification for Line~\ref{line:basic_condition} comes from Corollary~\ref{cor:naive_condition_g1}. If the stated necessary condition is not fulfilled, then $\calSoper(\PFsf)=\emptyset$, which is precisely the set returned by the algorithm. 
As Algorithm~\ref{alg:ForcedIntegers} returns \emph{fail} precisely when it is found some forced gap that is at the same time a forced element, Proposition~\ref{prop:intersection_gaps_elements} assures us that $\calSoper(\PFsf)=\emptyset$, which is the set returned by the algorithm.

When nothing of the above holds, the variable \textbf{semigroups} is initialized as the empty set and it is made a call to the recursive function \ref{func:RecursiveDepthFirstExploreTree}. As we are considering this variable global to the functions \ref{func:EndingCondition},
the result follows from Corollary~\ref{cor:recursive_function}.
\end{proof}
\section{Running times and examples}
\label{sec:running_times}
The number of semigroups can be quite large compared to the number of free elements. The following example illustrate this possibility.
\begin{example}\label{example:free_give_rise to_many}
\begin{verbatim}
gap>  pf := [ 68, 71, 163, 196 ];;
gap> forced := ForcedIntegersForPseudoFrobenius(pf);;
gap> free := Difference([1..Maximum(pf)],Union(forced));;
gap> Length(free);
38
gap> list := NumericalSemigroupsWithPseudoFrobeniusNumbers(pf);;
gap> Length(list);
1608
\end{verbatim}
In the continuation of previous \textsf{GAP} session we do a kind of verification of the result.
\begin{verbatim}
gap> ns := RandomList(list);
<Numerical semigroup>
gap> MinimalGeneratingSystem(ns);
[ 35, 38, 65, 81, 89, 94, 99, 101, 104, 106, 109, 110, 112, 113, 117, 118, 
  121, 122, 133 ]
gap> PseudoFrobeniusOfNumericalSemigroup(ns);
[ 68, 71, 163, 196 ]
\end{verbatim}
\end{example}
Table \ref{fig:execution_time_main} is meant to illustrate some timings. Its content was produced by using repeatedly the commands of the first part of the previous example. The candidates to $\PFsf$ were obtained randomly.
We observe that, although depending on some factors (such as the type or the Frobenius number we are interested in), the vast majority of the candidates would lead to the empty set. We do not consider them in this table (some were given in Table~\ref{fig:execution_time_failure}). 

\begin{table}
\begin{center}
    \begin{tabular}{| l | r | r | r |}
    \hline
pseudo Frobenius numbers &\# free &\# semigroups & time\\ \hline
[ 15, 27, 31, 43, 47 ]& 0& 1 &3\\ \hline
[ 16, 30, 33, 37 ]& 9& 3 &12\\ \hline 
[ 40, 65, 80, 89, 107, 110, 130 ]& 5& 3 &29\\ \hline
[ 32, 35, 44, 45, 48 ]& 13& 7 &24\\ \hline
[ 40, 65, 89, 91, 100, 106 ]&24&9&99\\ \hline
[ 36, 50, 56, 57, 63 ]& 25& 39& 123\\ \hline
[ 43, 50, 52, 65 ]&35&213&605\\ \hline
[ 68, 71, 163, 196 ]&38&1608&16603\\ \hline 
[ 38, 57, 67, 74, 79 ]& 40& 155 &527\\ \hline
[ 68, 72, 76, 77 ]& 46& 177 &607\\ \hline
[ 62, 78, 99, 129, 130 ]& 53& 4077 &28622\\ \hline
[ 128, 131, 146, 151, 180, 216, 224, 267, 271, 287 ]& 54& 954 &24253\\ \hline 
[ 84, 103, 144, 202, 230, 242, 245 ]& 56& 14292& 277094\\ \hline 
[ 66, 85, 86, 92 ]&55&950&4683\\ \hline
[ 76, 79, 88, 102 ]&64&1409&6505\\ \hline
[ 61, 67, 94, 105 ]&69&4432&21471\\ \hline
[ 114, 150, 179, 182, 231, 236, 254, 321 ]& 69& 302929& 7121020\\ \hline 
[ 62, 73, 166, 190, 203 ]&77&9934&134554\\ \hline
[ 102, 104, 118, 123, 134, 146, 149 ]&87&15910&149910\\ \hline
   \end{tabular}
\end{center}
\caption{Some examples of execution data of the main algorithm.\label{fig:execution_time_main}}
\end{table}

\section{Random}
\label{sec:random}
Sometimes one may just be interested in obtaining one numerical semigroup with \PFsf\ as set of pseudo-Frobenius numbers.
Algorithm~\ref{alg:NumericalSemigroupsWithPseudoFrobeniusNumbers} may be too slow (it gives much information that will not be used).
One could adapt the algorithm to stop once it encounters the first semigroup, but the information had to be transmitted recursively and one would end up with a slow algorithm.
Next we propose an alternative (Algorithm~\ref{alg:RandomNumericalSemigroupWithPseudoFrobeniusNumbers}), the first part of which is similar to the initial part of Algorithm~\ref{alg:NumericalSemigroupsWithPseudoFrobeniusNumbers}. The main difference is in the usage of the function \textrm{AnIrreducibleNumericalSemigroupsWithFrobeniusNumber} instead of the function \textrm{IrreducibleNumericalSemigroupsWithFrobeniusNumber}, both available in the \textsf{numericalsgps} package.
For the second part, instead of calling the recursive function, it tries to gess a path that leads to a leaf. Starts choosing at random a free integer $v$ and tests its non admissibility (by checking whether \SimpleForcedIntegers returns \fail when called with $v$ as if it was forced). If one does not conclude that $v$ is non admissible, it is assumed to be a forced integer.
There is an option that is part of the implementation to give a bound for the maximum number of attempts the function does. Its usage is ilustrated in Examples~\ref{example:random1} and~\ref{example:random2}.
\begin{algorithm}[ht]\caption{RandomNumericalSemigroupWithPseudoFrobeniusNumbers\label{alg:RandomNumericalSemigroupWithPseudoFrobeniusNumbers}}
\Input{\PF,\ max\_attempts}
\Output{One numerical semigroup S (at random) such that $\PFfunction(S)=\PF$ if it discovers some, \fail if it discovers that no semigroup exists... Otherwise, suggests to use more attempts}
\If{$(\type = 1) \Or (\type = 2 \And \PF[1]=\PF[2]/2)$}{\Return \textrm{AnIrreducibleNumericalSemigroupsWithFrobeniusNumber}(\frob);}
\If{$\Not\ \PF[\type] - \PF[\type-1] > \PF[1])$}{\Return \fail;}
$f\_ints := \ForcedIntegers(\PF)$\;
\If{f\_ints = \fail}{\Return \fail;}
$free := \Difference([1..\frob],\Union(f\_ints))$\;
\For{$i\in[1..max\_attempts]$}{
   \While{$free\ne\emptyset$}{
      $v := \RandomList(free)$\;
      $nfig := \SimpleForcedIntegers(\Union(f\_ints[1],[v]),f\_ints[2])$\;
      $nfie := \SimpleForcedIntegers(f\_ints[1],\Union(f\_ints[2],[v]))$\;
      \eIf{$nfig\ne \fail$}{\eIf{$\IsRange(\Union(nfig))$}{\Return $\NumericalSemigroupByGaps(nfig[1])$;}{$free := \Difference([1..\frob],\Union(nfig))$;}}
      {\eIf{$nfie\ne \fail$}{\eIf{$\IsRange(\Union(nfie))$}{\Return $\NumericalSemigroupByGaps(nfie[1])$;}{$free := \Difference([1..\frob],\Union(nfie))$;}}{\xbreak;}}
}}
\Comment{Info: Increase the number of attempts...}
\end{algorithm}

It may happen that no semigroup has the given set as set of pseudo-Frobenius elements, and thus the output will simply be \fail.
\begin{example}\label{example:random1}
We look for a random numerical semigroup with $\PFsf=\{100,453,537,543\}$. The first execution of the function yields:
\begin{verbatim}
gap> pf := [ 100, 453, 537, 543 ];;
gap> ns := RandomNumericalSemigroupWithPseudoFrobeniusNumbers(pf);;time;
 MinimalGeneratingSystem(ns);
2440
[ 66, 94, 106, 126, 166, 184, 194, 206, 209, 216, 224, 230, 235, 246, 256, 263,
 267, 284, 295, 309, 363, 374, 379, 385, 391, 413 ]
\end{verbatim}
While if we execute it a second time we get another semigroup with the desired set of pseudo-Frobenius numbers (clearly, the reader might obtain different outputs):
\begin{verbatim}
gap> pf := [ 100, 453, 537, 543 ];;
gap> ns := RandomNumericalSemigroupWithPseudoFrobeniusNumbers(pf);;time; 
MinimalGeneratingSystem(ns);
7302
[ 94, 106, 123, 134, 162, 178, 184, 194, 204, 206, 222, 223, 234, 235, 238, 248,
 251, 262, 263, 266, 270, 276, 279, 283, 293, 304,  313, 336, 348, 367, 383, 415 ]
\end{verbatim}
\end{example}
\begin{example}\label{example:random2}
If one of the free integers can neither be a gap nor an element, no semigroup exists.
\begin{verbatim}
gap> pf := [ 30, 104, 118, 147, 197, 292, 298, 315, 333, 384, 408 ];;
gap> ns := RandomNumericalSemigroupWithPseudoFrobeniusNumbers(       
>  rec(pseudo_frobenius := pf, max_attempts := 100));time;        
fail
22
\end{verbatim}
\end{example}
\appendix
\section{An approach based on irreducible numerical semigroups}
\label{sec:computing_arbitrary_type}

We present here an alternative way to compute the set of numerical semigroups with a given set of pseudo-Frobenius numbers. In general, this procedure is slower than the presented above, though we have not been able to characterize when this happens. We include it the manuscript since it was the initial implementation and was used to test the other one.

A numerical semigroup is \emph{irreducible} if it cannot be expressed as the intersection of two numerical semigroups properly containing it. It turns out that a numerical semigroup $S$ is irreducible if and only if either $\PFoper(S)=\{\mathrm F(S)\}$ or $\PFoper(S)=\{\mathrm F(S)/2,\mathrm F(S)\}$ (see \cite[Chapter 3]{NS}). Irreducible numerical semigroups can be also characterized as those maximal (with respect to set inclusion) numerical semigroups in the set of all numerical semigroups with given Frobenius number.

The maximality of irreducible numerical semigroups in the set of all numerical semigroups with given Frobenius number implies that every numerical semigroup is contained in an irreducible numerical with its same Frobenius number. Actually, we can say more.

\begin{lemma}\label{lemma:puedo-cubrir-f2}
Let $S$ be a numerical semigroup. There exists an irreducible numerical semigroup $T$ such that
\begin{enumerate}
\item $\mathrm F(S)=\mathrm F(T)$,
\item $]\mathrm F(S)/2,\mathrm F(S)[\cap \PFoper(S)\subset T$.
\end{enumerate}
\end{lemma}
\begin{proof}
Let $f=\mathrm F(S)$ and let $F=]f/2,f[\cap \PFoper(S)$. We claim that $S'=S\cup F$ is a numerical semigroup with $\mathrm F(S')=f$. Take $s,s'\in S'\setminus\{0\}$. 
\begin{itemize}
\item If both $s$ and $s'$ are in $S$, then $s+s'\in S\subseteq S'$.
\item If $s\in S$ and $s'\in F$, then $s+s'\in S$, because $s'\in \PFoper(S)$.
\item If $s,s'\in F$, then $s+s'>f$, and so $s+s'\in S\subseteq S'$.
\end{itemize}
Let us show that $\mathrm F(S')=f$. Assume to the contrary that this is not the case, and consequently $f\in S'$. Then as all the elements in $F$ are greater than $f/2$, and $f\not\in S$, there must be an element $s\in S$ and $g\in \PFoper(S)$ such that $g+s=f$. But this is impossible, since all elements in $\PFoper(S)$ are incomparable with respect to $\le_S$ (Lemma~\ref{lemma:maximals}).

If $S'$ is not irreducible, then as irreducible numerical semigroups are maximal in the set of numerical semigroups with fixed Frobenius number, there exists $T$ with $\mathrm F(S)=\mathrm F(S')=\mathrm F(T)$ and containing $S'$; whence fulfilling the desired conditions.
\end{proof}

With all these ingredients we get the following result.

\begin{proposition}\label{prop:chain}
Let $S$ be a nonirreducible numerical semigroup with $\PFoper(S)=\{g_1<\dots <g_k\}$, $k\ge 2$. Then there exists a chain
\[S=S_0\subset S_1=S\cup\{x_1\}\subset \dots \subset S_l=S\cup\{x_1,\ldots,x_l\},\]
with
\begin{enumerate}
\item $S_l$ irreducible,
\item $x_i=\max (S_l\setminus S_{i-1})$ for all $i\in\{1,\ldots, l\}$,
\item $(g_k/2,g_k)\cap \PFoper(S)\subset S_l$, 
\item for every $i\in \{1,\ldots, l\}$, $x_i$ is a minimal generator of $S_i$ such that $g_j-x_i\in S_i$ for some $j\in \{1,\ldots,k\}$,
\item for every $i\in \{1,\ldots, l\}$ and $f\in \PFoper(S_i)$ with $f\neq g_k$, there exists $j\in \{1,\ldots,k-1\}$ such that $g_j-f\in S_i$.
\end{enumerate}
\end{proposition}
\begin{proof}
Let $T$ be as in Lemma~\ref{lemma:puedo-cubrir-f2}. Construct a chain joining $S$ and $T$ by setting $S_0=S$ and $S_i=S_{i-1}\cup \{x_i\}$, with $x_i=\max(T\setminus S_{i-1})$. Then $S_i$ is a numerical semigroup, and $x_i$ is a minimal generator of $S_i$ and a pseudo-Frobenius number for $S_{i-1}$ (\cite[Lemma 4.35]{NS}). Since the complement $T\setminus S$ is finite, for some $k$, $S_k=T$. 

Clearly, $x_1=g_k$, and thus $x_1-g_k=0\in S$. Now let $i\in \{2,\ldots,k\}$. Then $x_i\in T\setminus S$, and by Lemma~\ref{lemma:maximals}, there exists $j\in\{1,\ldots, k\}$ such that $g_j-x_i\in S\subseteq S_i$. 

Take $f\in \PFoper(S_i)\setminus\{g_k\}$. Then $f\not\in S$ and thus $g_j-f\in S$ for some $j\in \{1,\ldots,k\}$ (Lemma~\ref{lemma:maximals} once more). Consequently $g_j-f\in S_i$. Notice that $j\neq k$, since $f-g_k<0$.
\end{proof}

Given a candidate set \PFsf\ of pseudo-Frobenius numbers with maximum element $f$, we can use the above procedure to construct from the set of all irreducible numerical semigroups with Frobenius number $f$, the set of all numerical semigroups having \PFsf\ as a set its pseudo-Frobenius numbers. In order to compute the set of all irreducible numerical semigroups with Frobenius number $f$ we use implementation of the procedure presented in \cite{br} that is already part of \cite{numericalsgps}. We have slightly modified the algorithm in \cite{br} to compute the set of irreducible numerical semigroups containing a given set of integers, and these integers are the first component of  \ForcedIntegers(\PFsf). For every irreducible element in the list we then remove those minimal generators fulfilling condition (4) in Proposition \ref{prop:chain}. We add to our list of semigroups the semigroups obtained in the preceding step for which condition (5) holds, and then we proceed recursively.

\begin{example}
Let us illustrate the above procedure with $\PFsf=\{10,13\}$. The number of irreducible numerical semigroup with Frobenius number $13$ is $8$. However, if we first call $\ForcedIntegers$ we get:

\begin{verbatim}
gap> ForcedIntegersForPseudoFrobenius([10,13]);
[ [ 1, 2, 3, 5, 6, 10, 13 ], [ 0, 7, 8, 11, 12, 14 ] ]
\end{verbatim}

Since we have modified the function \texttt{IrreducibleNumericalSemigroupsWithFrobeniusNumber} to output only those irreducible numerical semigroups containing $\{0,7,8,11,12,14\}$, we obtain only two irreducible numerical semigroups: $S_1=\left< 4, 7, 10 \right>$ and $S_2=\langle   7, 8, 9, 10, 11, 12 \rangle$.

For $S_1$ the only minimal generator that fulfills the conditions in Proposition \ref{prop:chain} is $10$. If we remove $10$ from $S_1$, we obtain $T_1=\langle 4,7,17\rangle$, which already has the desired set of pseudo-Frobenius number. 

As for $S_2$, again $10$ is the only minimal generator fulfilling the conditions in Proposition \ref{prop:chain}, and we obtain $T_2=\langle 7, 8, 9, 11, 12\rangle$. This semigroup has pseudo-Frobenius number set equal to $\PFsf$, and so, as with $T_1$ we do not need to look for new minimal generators to remove.

Thus, the only numerical semigroups with pseudo-Frobenius number set $\{10, 13\}$ are $T_1$ and $T_2$.
\end{example}

%
%
%
%
%
%
%




\begin{thebibliography}{100}

\bibitem{br} V. Blanco, J. C. Rosales, The tree of irreducible numerical semigroups with fixed Frobenius number, Forum Math. \textbf{25} (2013), 1249--1261.


\bibitem{bres-sym} H. Bresinsky,  Symmetric semigroups of integers generated by 4 elements,  Manuscripta Math. \textbf{17} (1975), 205-219.

\bibitem{intpic} M.~Delgado,  \lq\lq intpic\rq\rq, \emph{a \textsf{GAP} package for drawing integers}. Available via \url{http://www.gap-system.org/}.

\bibitem{numericalsgps} M. Delgado, P. A. Garc\'{\i}a-S\'{a}nchez and J. Morais, \lq\lq NumericalSgps\rq\rq, \emph{a \textsf{GAP} package for numerical semigroups, Version 0.990}; 2015.  Available via \url{http://www.gap-system.org/}.

\bibitem{GAP4} The GAP~Group, \emph{GAP -- Groups, Algorithms, and Programming, Version 4.7.7}; 2015. Available via \url{http://www.gap-system.org/}.

\bibitem{komeda} J. Komeda, On the existence of Weierstrass points with a certain semigroup generated by 4 elements, Tsukuba J. Math. \textbf{6} (1982), 237-270. 

\bibitem{RoblesRosales_preprint:type2} A. M. Robles-P\'{e}rez and J. C. Rosales, The genus, the Frobenius number, and the pseudo-Frobenius numbers of numerical semigroups with type two, preprint.

\bibitem{NS} J. C. Rosales and P. A. Garc\'{\i}a-S\'{a}nchez, \lq\lq Numerical Semigroups\rq\rq, \emph{Developments in Maths.} \textbf{20},
  Springer (2010).

\bibitem{huecos-fun} J. C. Rosales, P. A. Garc\'{\i}a-S\'{a}nchez, J. I. Garc\'{\i}a-Garc\'{\i}a, and J. A. Jim\'{e}nez-Madrid, Fundamental gaps in numerical semigroups, J. Pure Appl. Algebra \textbf{189} (2004), 301--313.

\end{thebibliography}
\end{document}